\documentclass[a4paper,11pt]{article}
\usepackage[utf8]{inputenc}
\usepackage[english]{babel}
\usepackage[T1]{fontenc}
\usepackage[top=3.5cm, bottom=3.5cm, left=3cm, right=3cm]{geometry}
\usepackage{color}    
\usepackage{graphicx} 
\usepackage{listings}
\usepackage{tikz}
\usetikzlibrary{positioning}
\usepackage{fancyhdr} 
\usepackage{lastpage} 
\usepackage{amsmath,amssymb,mleftright} 
\usepackage{amsthm}
\usepackage{listings} 
\usepackage{parcolumns} 
\usepackage{array,tabu,booktabs} 
\usepackage{multirow} 
\usepackage{pgfplots,wrapfig}
\usepackage{titling}  
\usepackage{lmodern}  
\usepackage{verbatim,verbatimbox} 
\usepackage{fancyvrb} 
\usepackage{enumitem} 
\PassOptionsToPackage{hyphens}{url}\usepackage{hyperref} 
\usepackage{float}
\usepackage{caption}
\usepackage{subcaption}
\usepackage{animate}
\usepackage[final]{pdfpages}
\usepackage{diagbox}
\usepackage{hhline}
\usepackage{titlesec}
\usepackage{algpseudocode}
\usepackage[export]{adjustbox}
\usepackage{bigints}
\usepackage[makeroom]{cancel}
\usepackage{xparse}
\usepackage[ruled,vlined]{algorithm2e}
\usepackage{setspace}
\usepackage{mathtools}
\titlespacing*{\section}
{0pt}{5.5ex plus 1ex minus .2ex}{4.3ex plus .2ex}
\titlespacing*{\subsection}
{0pt}{5.5ex plus 1ex minus .2ex}{4.3ex plus .2ex}
\usepackage{rotating}

\usepackage{titlesec}

\titleformat{\paragraph}
{\normalfont\normalsize\bfseries}{\theparagraph}{1em}{}
\titlespacing*{\paragraph}
{0pt}{3.25ex plus 1ex minus .2ex}{1.5ex plus .2ex}

\makeatletter
\newcommand*\bigcdot{\mathpalette\bigcdot@{.5}}
\newcommand*\bigcdot@[2]{\mathbin{\vcenter{\hbox{\scalebox{#2}{$\m@th#1\bullet$}}}}}
\makeatother

\makeatletter
\def\BState{\State\hskip-\ALG@thistlm}
\makeatother
\definecolor{tempcolor}{rgb}{0.7773 0.0820 0.5195}

\newcommand{\lp}{\left(}
\newcommand{\rp}{\right)}

\newcommand{\norm}[1]{\ensuremath{\left\lVert  #1\right\rVert}} 

\newcommand{\divg}[1]{\ensuremath{\text{div}_g\left(  #1\right)}}

\usepackage{xcolor}
\definecolor{dkgreen}{rgb}{0,0.6,0}
\definecolor{dred}{rgb}{0.545,0,0}
\definecolor{dblue}{rgb}{0,0,0.545}
\definecolor{lgrey}{rgb}{0.9,0.9,0.9}
\definecolor{gray}{rgb}{0.4,0.4,0.4}
\definecolor{darkblue}{rgb}{0.0,0.0,0.6}
\definecolor{turquoise}{rgb}{0.2500,0.8750,0.8125}
\definecolor{indigo}{rgb}{0.2930, 0,0.5078}
\definecolor{mag}{rgb}{1, 0,1}
\definecolor{corn}{rgb}{0.3906,0.5820,0.9258}
\definecolor{mvr}{rgb}{0.7773,0.0820,0.5195}
\definecolor{dod}{rgb}{0.11719,0.5625,1}
\lstdefinelanguage{MatlabCostum}{
      backgroundcolor=\color{white},  
      basicstyle=\footnotesize \ttfamily \color{black} \bfseries,   
      breakatwhitespace=false,       
      breaklines=true,               
      captionpos=b,                   
      commentstyle=\color{dkgreen},   
      emph={repmat, ones},			
      keywordstyle=\color{blue},           
      escapeinside={\%*}{*)},                  
      frame=single,                  
      language=Matlab,                
      identifierstyle=\color{black},
      stringstyle=\color{blue},      
      numbers=left,                 
      numbersep=5pt,                  
      numberstyle=\tiny\color{black}, 
      rulecolor=\color{black},        
      showspaces=false,               
      showstringspaces=false,        
      showtabs=false,                
      stepnumber=1,                   
      tabsize=4,
      title=\lstname
}
\newtheorem{theorem}{Theorem}[section]

\newtheorem{proposition}[theorem]{Proposition}
\newtheorem{remark}[theorem]{Remark}
\newtheorem{conjecture}[theorem]{Conjecture}

\newcommand{\Div}{\operatorname{div}}

\newcommand{\Det}{\operatorname{det}}

\usepackage[colorinlistoftodos,prependcaption,textsize=tiny]{todonotes}

\allowdisplaybreaks

\usepackage[backend=biber,
            style=numeric,
            abbreviate=false,
            dateabbrev=false,
            alldates=long]{biblatex}

\begin{document}

\title{Reconstructing anisotropic conductivities\\ on two-dimensional Riemannian manifolds\\ from power density measurements}

\author{Kim Knudsen, Steen Markvorsen, Hjørdis Schlüter}
\maketitle

\begin{abstract}
We consider an electrically conductive compact two-dimensional Riemannian manifold with smooth boundary.
This setting defines a natural conductive Laplacian on the manifold and hence also voltage potentials, current fields and corresponding power densities arising from suitable boundary conditions. Motivated by Acousto-Electric Tomography we show that if the manifold has  genus zero and the metric is known, then the anisotropic conductivity can be recovered uniquely and constructively from knowledge of a few power densities. We illustrate the procedure numerically by reconstructing an anisotropic conductivity on the catenoid, i.e. the classical genus zero minimal surface in three-space.
\end{abstract}

\section{Introduction and statement of the main result}
Let $(M, g)$ denote a compact two-dimensional Riemannian manifold with smooth boundary $\partial M.$ An electric conductivity on $M$ is modelled by a -- generally anisotropic -- $(1,1)$ tensor field $\gamma$, which is selfadjoint and uniformly elliptic with respect to $g$, i.e. for some $\kappa >1$ and for all tangent vectors $v$ and $w$:
\begin{equation}
    g(\gamma(v),w) = g(v, \gamma (w)) \quad \textrm{and}\quad  \kappa^{-1} \,\Vert v \Vert_g^2 \leq g(\gamma (v), v) \leq \kappa \, \Vert v \Vert_g^2; \label{selfadjoint}
\end{equation}
$\,\Vert v \Vert_g$ denotes the metric induced norm on the tangent space.

On the boundary $\partial M$ we prescribe an electrostatic potential $f$ that generates an interior voltage potential $u.$ In the absence of interior sinks and sources, the potential $u$ is characterized as the unique solution to the boundary value problem 
\begin{align} \label{eqPDE}
    \Bigg\{\begin{split}
    \divg{\gamma \, \nabla_g{u}} &=  0  \text{ in } M,\\
     u &=  f \text{ on }\partial M.
    \end{split}
\end{align}
Existence and uniqueness of solutions to such an elliptic PDE on a Riemannian manifold is classical, see e.g. \cite{Salo2013}. The interior current field is $\gamma \, \nabla_g{u},$ i.e.\ $\gamma$ is the tensor turning the electric field $\nabla_g{u}$ into the current field. The conductivity tensor $\gamma$ can be visualized on $M$ by its corresponding ellipse field. At each point on $M$, the $(1,1)$-tensor $\gamma$ gives rise to an ellipse by its corresponding action as a linear map on the set of unit vectors in the tangent plane at the point; such an ellipse field is illustrated in figure \ref{fig:CondEllipseIntro}.


\begin{figure}[h!]
    \centering
    \begin{minipage}[t]{\textwidth}
        \begin{minipage}[t]{\textwidth}
        \centering
        \includegraphics[width=0.7\linewidth]{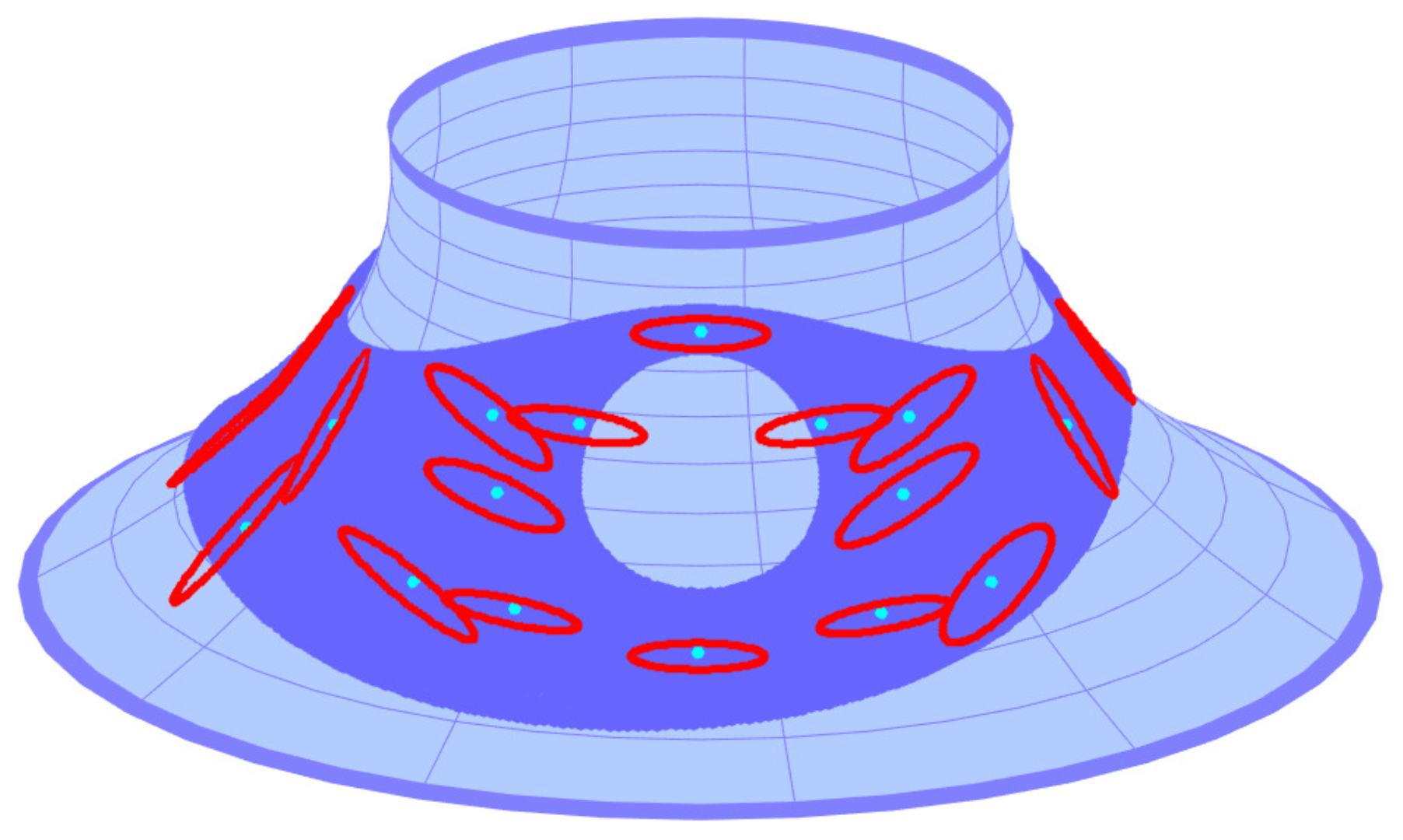}
    \end{minipage}
            \end{minipage}
    \caption{Example of a $\gamma$ induced ellipse field on a catenoid.}
    \label{fig:CondEllipseIntro}
    \end{figure}

By considering $m$ different boundary functions $f=f_{i},\; 1\leq i \leq m,$  the corresponding solutions to equation \eqref{eqPDE} are denoted by $u_{i}$. They define the so-called power density $(m \times m)$-matrix $H$ with elements:
\begin{equation}
    H_{ij}= H_{ji} = g(\gamma \,\nabla_g{u_i},\nabla_g{u_j}) \quad \text{for}\quad  1 \leq i,j \leq m. \label{Heq}
\end{equation}

The inverse problem in Acousto-Electric Tomography \cite{Zhang:Wang:2004,ammari2008a, jensen2019feasibility} is concerned with the above in a Euclidean setting and asks for the recovery of $\gamma$ from knowledge of $H.$ This problem was addressed in two dimensions in \cite{BalMonard,BalUhlmann2013,BalGuo2014,adesokan2019a} and for higher dimensions in \cite{MonardBal2013,BalBonnetier2013,Monard2018}. It is highly related to the Calder\'on problem \cite{Calderon1980}, see \cite{Uhlmann2013} for an account on this and related problems.

We propose the following:
\begin{conjecture} \label{Conj:General}
Suppose we know the $2$-dimensional compact Riemannian manifold $(M,g)$ as well as the power density matrix $H$ associated with a finite number of sufficiently well-chosen boundary functions $f_i.$ Then the conductivity tensor field $\gamma$ can be uniquely and constructively determined.
\end{conjecture}

As a first step towards this general conjecture, we prove it for all cases where the topology of the manifold is simplest possible in the following sense:

Every 2-dimensional, compact, orientable Riemannian manifold with $k$ boundary components is diffeomorphic to a handle-body, i.e. it consists of a number of handles attached to a 2-sphere with $k$ disks removed. The number of handles $p$ is called the \emph{genus} of the manifold. It is related to the Euler characteristic $\chi(M)$ and $k$ as follows: $p = 1 - \left( \chi + k \right)/2$, see \cite{hirsch1976a}.

With these preliminaries, we can now state our main result. 
\begin{theorem} \label{MainThm}
Let $(M,g)$ denote a given compact $2$-dimensional Riemannian manifold with genus $p = 0$, metric $g$ and non-empty smooth boundary $\partial M$ consisting of a finite number of boundary components. Suppose that $\gamma$ is a conductivity tensor field on $M$, which is known on the boundary $\partial M$. Then there exist $m=4$ boundary functions $f_{i}$ with induced corresponding power density matrix $H,$ such that $H$ determines $\gamma$ uniquely and constructively on all of $M.$
\end{theorem}

Our approach takes in particular advantage of the work in 2D by Monard and Bal \cite{BalMonard}. In addition, the proof makes use of the Poincar\'{e}-Koebe uniformization theorem for compact Riemann surfaces (with boundary). Indeed, $(M, g)$ admits a global conformal parametrization from a \emph{fundamental domain} in $\mathbb{R}^{2}$. 

When the genus is $p = 0,$ the fundamental domain is an open set in $\mathbb{R}^{2}$ with the entire boundary accessible for specifications of boundary functions. When $p > 0,$ however, the fundamental domain has extra boundary components that must be glued together to represent the original manifold. This is for instance the case for the torus ($p=1$). We refer to \cite{gu2002a,gu2004a,schoen1997a} for a very general and constructive approach to this uniformization procedure. 


The four boundary functions mentioned in Theorem \ref{MainThm} stem from the proof techniques of Monard and Bal \cite{BalMonard}, who formulate the conditions \eqref{cond11M}-\eqref{cond12M} below for the corresponding solutions. They show the existence of four complex geometrical optics (CGO) solutions that work, however, the CGO solutions depend on the particular unknown conductivity, and hence they are impractical for solving the inverse problem. For that reason, we use tailored polynomials instead. We will use the same approach in our numerical tests. Whether the number of four conditions is optimal or the solutions can be chosen independently from the conductivity, are open questions. 

In the case of higher genus, such CGO solutions are not known to exist, and four conditions may not suffice. Other techniques e.g. based on Runge approximation, can likely yield an upper bound; we leave this question for further studies. See \cite{alberti2022a} for recent probabilistic work regarding the choice of boundary conditions satisfying a Jacobi condition similar to \eqref{cond11M}.

\begin{remark}
We note that the manifolds mentioned in Theorem \ref{MainThm} are very general in the sense that they do not need to be realizable as surfaces in $\mathbb{R}^{3}$ (with the induced metric).
\end{remark}

The main novelty of our work is  the geometric setting and the application of the uniformization theorem, which allows for the reduction to the known, Euclidean setting via the conformal parametrization of the 2D manifold. The outline of the paper is as follows: We prove Theorem \ref{MainThm} in Section \ref{secMainNew}. In that section we review some details from the Euclidean reconstruction method and indicate how the key equations can be lifted to $(M, g)$ using explicitly the conformal factor that is induced by the conformal parametrization of $(M, g)$. In Section \ref{secChoice} we introduce other parametrizations of the manifold, and analyse how the method is affected. Section \ref{secthmBM} is devoted to the formulation of the reconstruction method, and finally, in Section \ref{secNumExample}, we illustrate the result by a numerical implementation of the reconstruction procedure on a specific compact subset of the catenoid.

\section{Reduction to the Euclidean setting}
\label{secMainNew}
The proof of Theorem \ref{MainThm} relies on the fact that the conductivity equation on a conformally parametrized manifold is directly related to the corresponding equation in the Euclidean plane. Moreover, the power density matrix transforms in a similar and straighforward way. These facts reduce the reconstruction problem on a conformally parametrized manifold  to the well-known reconstruction problem in the Euclidean plane.

We use standard coordinates $x^{1}$ and $x^{2}$ and the corresponding standard basis $\{\frac{\partial}{\partial {x^{1}}} = e_{1}\, , \,\,\frac{\partial}{\partial {x^{2}}} = e_{2}\}$ in $\mathbb{R}^{2}$.

The manifold $(M, g)$ is represented by a 
fundamental domain $M\subset \mathbb{R}^{2}$ and a metric $g$ that is conformal to the Euclidean metric $g_{E}$ via a corresponding conformal factor $\rho$: 
\begin{equation}\label{eqMetricConformal}
g(e_{i}, e_{j}) = \rho^{2}\cdot g_{E}(e_{i}, e_{j}) = \rho^{2} \cdot \delta_{i}^{j} \quad 
\end{equation}
with $\delta_{i}^{j}$ denoting the Kronecker delta. 
%
%
%
%
The metric $g$ is then represented by a simple diagonal $(2 \times 2)$-matrix function $G$ with equal diagonal elements, and 
the conductivity tensor is represented by a $(2 \times 2)$-matrix function with elements $(\gamma)_i^{j}$ so that $\gamma(e_{i})(x^{1}, x^{2}) = \sum_{j}(\gamma)_i^{j}(x^{1}, x^{2})\cdot e_{j}$.

The first equation in \eqref{eqPDE} is called the $\gamma$-Laplace equation and can be written shorthand as $\Delta_g^{\gamma} u = 0$. The following observation clarifies how the equation change under a conformal change of the metric:


\begin{proposition} \label{propConformal}
Let $u(x^{1}, x^{2})$ denote a smooth function on $M$. Then
\begin{equation} \label{eqDelta}
\Delta_{g}^{\gamma}u(x^{1}, x^{2}) = \frac{1}{\rho^{2}(x^{1}, x^{2})} \cdot \Delta_{g_{E}}^{\gamma}u(x^{1}, x^{2}).
\end{equation}
\end{proposition}
\begin{proof}
The $g$-gradient of the function $u$ is expressed using the elements $G^{ij}$ of the inverse matrix $G^{-1}$ as follows:
\begin{equation*}
\nabla_g u = \sum_{i\, j}\frac{\partial u}{\partial x^{i}}\cdot G^{j\,i} \cdot e_{j} = \frac{1}{\rho^{2}}\cdot \sum_{i}\frac{\partial u}{\partial x^{i}}\cdot e_{i}  = \frac{1}{\rho^{2}} \cdot \nabla_{g_{E}} u,
\end{equation*}
and the
$g$-divergence of a vector field $V = \sum_{i} v^{i}\cdot e_{i}$ is in $2$D:
\begin{equation*} \label{3:eqShortDiv}
\Div_g(V) =  \frac{1}{\sqrt{\Det(G)}}\cdot \sum_{i}\frac{\partial}{\partial x^{i}}\left(v^{i}\cdot \sqrt{\Det(G)}\right) =  \frac{1}{\rho^{2}}\cdot \sum_{i}\frac{\partial}{\partial x^{i}}\left(v^{i}\cdot \rho^{2}\right).
\end{equation*}
Insertion of $V = \gamma(\nabla_g(u)) = \frac{1}{\rho^{2}} \cdot \gamma(\nabla_{g_{E}} u)$ now gives directly:
\begin{equation*} \label{eqDiv2}
\Div_g(\gamma(\nabla_g u)) = \frac{1}{\rho^{2}} \cdot \Div_{g_{E}}(\gamma(\nabla_{g_{E}}u)),
\end{equation*}
and the proposition follows.
\end{proof}

This proposition implies that the solution to the boundary value problem \eqref{eqPDE} is the same as the solution to the problem in the Euclidean plane:
\begin{align} \label{eqPDENE}
    \Bigg\{\begin{split}
    \Delta_{g_E}^{\gamma} u &=  0  \text{ in } M,\\
     u &=  f \text{ on }\partial M.
    \end{split}
\end{align}
In the following, let $H^E$ denote the power density matrix defined in terms of solutions to \eqref{eqPDENE}. The second observation concerns the change in the power density matrix under a conformal change of the metric:
\begin{proposition} \label{propConformal2}
The power density matrix for the conformal metric $g$ satisfies
\begin{equation} \label{eqHconform}
H_{ij} = \frac{1}{\rho^{2}}\cdot H^{E}_{ij} .
\end{equation}
\end{proposition}
\begin{proof}
    Since the solution to \eqref{eqPDE} and \eqref{eqPDENE} are the same, the result follows directly by calculation:
    \begin{equation*}
\begin{aligned}
H_{ij} = H_{ji} &= g(\gamma \nabla_{g} u_{i}, \nabla_{g} u_{j}) \\ &= \frac{1}{\rho^{2}}\cdot g_{E}(\gamma \nabla_{g_E} u_{i}, \nabla_{g_E} u_{j}) \\
&=
\frac{1}{\rho^{2}}\cdot H^{E}_{ij} \quad , \quad \textrm{for} \quad 1 \leq i\,, j \, \leq m.
\end{aligned}
\end{equation*}
\end{proof}
The proof of Theorem \ref{MainThm} can now be completed:
\begin{proof}[Proof of Theorem \ref{MainThm}]
    The purely Euclidean problem of reconstructing $\gamma$ from $H^E$ was considered by Monard and Bal \cite[Theorem 2.2]{BalMonard}. Indeed, they show the existence of $m=4$ conditions $f_i$ that make their method valid.
    
    For the geometric problem, we use the same four boundary conditions for \eqref{eqPDE} to obtain first the data $H;$ then we calculate $H^E$ by \eqref{eqHconform}, and finally $\gamma$ by the mentioned reconstruction algorithm.
\end{proof}


\begin{remark}
The identity \eqref{eqDelta} only holds in dimension $2$ where we can use the fact that $\sqrt{\Det(G)} = \rho^{2}$ -- as is already well-known from the case of an isotropic conductivity $(\gamma)_{i}^{j}(x^1, x^{2}) = q(x^{1}, x^{2})\cdot\delta_{i}^{j}$ for a positive function $q$.
\end{remark}

We make a final observation regarding the conditions for the four boundary potential functions and the corresponding solutions to \eqref{eqPDENE}. According to Monard and Bal \cite{BalMonard}, the boundary functions must be chosen such that the interior solutions satisfy 
    \begin{align}
        \min(\text{det}(\nabla_{g_E} u_{1},\nabla_{g_E} u_{2}), \text{det}(\nabla_{g_E} u_{3} ,\nabla_{g_E} u_{4}))\geq c_E >0  \,\, \text{for every }x \in M,\label{cond11E}\\
        \nabla_{g_E} \left(\log \left( \frac{\text{det}(\nabla_{g_E} u_{1},\nabla_{g_E} u_{2})}{ \text{det}(\nabla_{g_E} u_{3},\nabla_{g_E} u_{4})}\right)\right) \neq 0 \,\, \text{for every }x \in M. \label{cond12E}
    \end{align}
The argument for the existence of such four boundary conditions relies on complex geometrical optics solutions. However, in practical computations these solutions are not known on $\partial M,$ and therefore other families of boundary functions, e.g., polynomials, are used. We will return to this in Section \ref{secNumExample}.

The geometric counterparts of \eqref{cond11E}--\eqref{cond12E} now read as follows:
\begin{proposition}\label{prop:RelatingConditions}
    The solutions $u_j,\; j=1,\ldots,4,$ satisfy \eqref{cond11E}--\eqref{cond12E} if and only if
    \begin{align}
        \min(\mathrm{det}(\nabla_{g} u_{1},\nabla_{g} u_{2}), \mathrm{det}(\nabla_{g} u_{3} ,\nabla_{g} u_{4}))\geq c = \rho^{-4} c_E >0,  \,\, x \in M,\label{cond11M}\\
        \nabla_{g} \left(\log \left( \frac{\mathrm{det}(\nabla_{g} u_{1},\nabla_{g} u_{2})}{ \mathrm{det}(\nabla_{g} u_{3},\nabla_{g} u_{4})}\right)\right) \neq 0, \,\, x \in M. \label{cond12M}
    \end{align} 
\end{proposition}
\begin{proof}
       The left hand sides in the conditions then only differ from the left hand side of the planar conditions \eqref{cond11E}--\eqref{cond12E} with respect to the conformal factor $\rho$:
    \begin{align*}
        &\det(\nabla_{g} u_{i},\nabla_{g} u_{j})(x)= \rho^{-4}\det(\nabla_{g_E} u_{i},\nabla_{g_E} u_{j})(x), \quad (i,j)\in \{(1,2),(3,4)\},\\
        &\nabla_{g} \left(\log \left( \frac{\text{det}(\nabla_{g} u_{1},\nabla_{g} u_{2})(x)}{ \text{det}(\nabla_{g} u_{3},\nabla_{g} u_{4})(x)}\right)\right) =\rho^{-2}\nabla_{g_E} \left(\log \left( \frac{\text{det}(\nabla_{g_E} u_{1},\nabla_{g_E} u_{2})(x)}{ \text{det}(\nabla_{g_E} u_{3},\nabla_{g_E} u_{4})(x)}\right)\right).
    \end{align*}
    Hence, as $\rho$ is positive, $u_{i}$ satisfy the conditions \eqref{cond11E}-\eqref{cond12E} if and only if $u_{i}$ satisfy the conditions \eqref{cond11M}-\eqref{cond12M}.
\end{proof}

\section{Choice of coordinates for the reconstruction procedure}\label{secChoice}
By the uniformization theorem we are guaranteed existence of global conformal coordinates, but in practice it is not a straightforward  procedure to find these coordinates. The authors in \cite{gu2002a,gu2004a,schoen1997a} give a constructive approach to obtain conformal coordinates given a choice of coordinates that are non-conformal. Therefore in order to make the reconstruction procedure in this paper constructive one has to make a choice of coordinates to begin with and assume that this choice not necessarily gives the desired conformal coordinates. The non-conformal parametrization of the manifold is denoted by $(N,g_N)$ and the conformally parametrized manifold is denoted by $(M,g_M)$. In the following we want to highlight the correspondence between the PDEs and the power density matrix under the coordinate transformation from $(N,g_N)$ and $(M,g_M)$ and show how they relate to a problem in the Euclidean plane.\\

We denote the metric matrix functions for $g_M$ and $g_N$ by $G_M$ and $G_N,$ respectively. Let $\psi:  N \ni y \mapsto \psi(y)=x \in M$  be the conformal diffeomorphism from $(N,g_N)$ to $(M,g_M).$ Then as $\psi$ is conformal, the pushforward of $G_N$ by $\psi$ satisfies:
\begin{equation*}
    G_M(x) = \psi_* G_N(x) = (D\psi^{-1})^t(x) G_N(\psi^{-1}(x)) D\psi^{-1}(x)=\rho^2(x) \cdot G_E,
\end{equation*}
where $(D\psi^{-1})^t(x)$ and $D\psi^{-1}(x)$ denote the transpose of the Jacobi matrix and the Jacobi matrix with respect to $\psi^{-1}$. The boundary value problem on $(N,g_N)$ is on the form
\begin{align} \label{eqPDEM}
    \Bigg\{\begin{split}
    \Delta_{g_N}^{\gamma_N} u_N(y) &=  0  \text{ in } N,\\
     u_N(y) &=  f_N(y) \text{ on }\partial N,
    \end{split}
\end{align}
This can be expressed in the coordinates $x=\psi(y)$ on $(M,g_M)$ as follows:
\begin{align} \label{eqPDEN}
    \Bigg\{\begin{split}
    \Delta_{g_M}^{\psi_* \gamma_N} u_M(x) &=  0  \text{ in } M,\\
     u_M(x) &=  f_M(x) \text{ on }\partial M,
    \end{split}
\end{align}
where $u_M(x)=u_N(\psi^{-1}(x))$, $f_M(x)=f_N(\psi^{-1}(x))$ and $\psi^* \gamma_N(x)$ is the pushforward by $\psi$ of the conductivity tensor $\gamma_N(y)$. The latter is defined as
\begin{equation*}
    \gamma_M(x) =\psi_* \gamma_N(x)= D \psi(\psi^{-1}(x)) \gamma_N(\psi(y)) D\psi^{-1}(x).
\end{equation*}
Using the expression for the pushforward of the metric matrix function $G_N$ and the chain rule one obtains the following relationship between the gradients:
\begin{align*}
    \nabla_{g_N} u_N(y) &= G_N^{-1} \begin{bmatrix}
        \partial/\partial y^1\\
        \partial/\partial y^2
    \end{bmatrix} u_M(\psi(y))\\ &= D\psi^{-1} G_M^{-1} (D\psi^{-1})^t (D\psi)^t \begin{bmatrix}
        \partial/\partial x^1\\
        \partial/\partial x^2
    \end{bmatrix} u_M(x)\\ &= D\psi^{-1} \nabla_{g_M} u_M(x).
\end{align*}

We note that for the reconstruction procedure in the plane it is essential that the power densities satisfy conditions \eqref{cond11E}-\eqref{cond12E}. The corresponding conditions are defined as follows on $(N,g_N)$:
\begin{align}
   \min(\text{det}(\nabla_{g_N} u_{N,1},\nabla_{g_N} u_{N,2}), \text{det}(\nabla_{g_N} u_{N,3} ,\nabla_{g_N} u_{N,4}))\geq c_N >0  \,\, \text{for every }y \in N, \label{cond11N1}\\
    \nabla_{g_N} \left(\log \left( \frac{\text{det}(\nabla_{g_N} u_{N,1},\nabla_{g_N} u_{N,2})}{ \text{det}(\nabla_{g_N} u_{N,3},\nabla_{g_N} u_{N,4})}\right)\right) \neq 0 \,\, \text{for every }y \in N. \label{cond12N1}
\end{align}
and they are on the form \eqref{cond11M}-\eqref{cond12M} on $(M,g_M)$.

 They follow the following relationship under the coordinate transformation:
     \begin{align}
        &\det(\nabla_{g_N} u_{N,i},\nabla_{g_N} u_{N,j})(y) \nonumber\\
        &\qquad= J_{\psi}^{-1}(\psi(y)) \det(\nabla_{g_M} u_{M,i},\nabla_{g_M} u_{M,j})(\psi(y)), \quad (i,j)\in \{(1,2),(3,4)\}\label{relMN1}\\
        &\nabla_{g_N} \left(\log \left( \frac{\text{det}(\nabla_{g_N} u_{N,1},\nabla_{g_N} u_{N,2})(y)}{ \text{det}(\nabla_{g_N} u_{N,3},\nabla_{g_N} u_{N,4})(y)}\right)\right) \nonumber\\ 
        &\qquad =D \psi^{-1}(\psi(y)) \nabla_{g_M} \left(\log \left( \frac{\text{det}(\nabla_{g_M} u_{M,1},\nabla_{g_M} u_{M,2})(\psi(y))}{ \text{det}(\nabla_{g_M} u_{M,3},\nabla_{g_M} u_{M,4})(\psi(y))}\right)\right), \label{relMN2}
    \end{align}
where $J_{\psi}^{-1}(y)=\det(D\psi)^{-1}$ and we note that $J_{\psi}\geq c>0$ since $D \psi$ is everywhere invertible. Combining these results yields the following proposition that links the conditions \eqref{cond11N1}-\eqref{cond12N1} and \eqref{cond11E}-\eqref{cond12E}:

\begin{proposition}
    The solutions $u_{M,j},\; j=1,\ldots,4,$ satisfy \eqref{cond11E}--\eqref{cond12E} if and only if the solutions $u_{j,N},\; j=1,\ldots,4,$ satisfy 
    \begin{align}
        &\min(\mathrm{det}(\nabla_{g_N} u_{N,1},\nabla_{g_N} u_{N,2})(y), \mathrm{det}(\nabla_{g_N} u_{N,3} ,\nabla_{g_N} u_{N,4})(y)) &&\nonumber\\
        &\hspace{30mm} \geq c_N = J_{\psi}^{-1}(\psi(y)) \rho^{-4}(\psi(y))  c_E >0,  & &y \in N,\label{cond11N}\\
        &\nabla_{g_N} \left(\log \left( \frac{\mathrm{det}(\nabla_{g_N} u_{N,1},\nabla_{g_N} u_{N,2})}{ \mathrm{det}(\nabla_{g_N} u_{N,3},\nabla_{g_N} u_{N,4})}\right)\right) \neq 0, & &y \in N. \label{cond12N}
    \end{align} 
\end{proposition}
\begin{proof}
       Combining proposition \ref{prop:RelatingConditions} and the observations in \eqref{relMN1} and \eqref{relMN2} yields that the left hand sides in the conditions only differ from the left hand side of the planar conditions \eqref{cond11E}--\eqref{cond12E} with respect to the conformal factor $\rho$ and the Jacobian of $\psi$:
    \begin{align*}
        &\det(\nabla_{g_N} u_{N,i},\nabla_{g_N} u_{N,j})(y)= J_{\psi}^{-1}(\psi(y))\rho^{-4}(\psi(y))\det(\nabla_{g_E} u_{M,i},\nabla_{g_E} u_{M,j})(y),  
        \intertext{for $(i,j)\in \{(1,2),(3,4)\}$, and}
        &\nabla_{g_N} \left(\log \left( \frac{\text{det}(\nabla_{g_N} u_{N,1},\nabla_{g_N} u_{N,2})(y)}{ \text{det}(\nabla_{g_N} u_{N,3},\nabla_{g_N} u_{N,4})(y)}\right)\right) \\
        &\qquad =D\phi^{-1}(\psi(y))\rho^{-2}(\psi(y))\nabla_{g_E} \left(\log \left( \frac{\text{det}(\nabla_{g_E} u_{M,1},\nabla_{g_E} u_{M,2})(\psi(y))}{ \text{det}(\nabla_{g_E} u_{M,3},\nabla_{g_E} u_{M,4})(\psi(y))}\right)\right).
    \end{align*}
    Hence, as $\rho$ is positive, $D\psi$ is invertible with positive determinant $J_{\psi}$, $u_{M,i}$ satisfy the conditions \eqref{cond11E}-\eqref{cond12E} if and only if $u_{N,i}$ satisfy the conditions \eqref{cond11N}-\eqref{cond12N}.
\end{proof}

\section{The reconstruction procedure}
\label{secthmBM}
In the first part of this section we summarize the Euclidean reconstruction procedure based on \cite{BalMonard}, see also \cite{schlueter2022a}. In the second part we use this to obtain the reconstruction procedure for conformally parameterized manifolds and then generalize the procedure to non-conformally parametrized manifolds.

\subsection{The Euclidean reconstruction procedure in $(M,g_E)$}
\label{secErec}
Based on (the unknown) $\gamma$ we introduce another smooth (1,1) tensor field $A$ defined in each tangent space as $A^2(x)=\gamma(x)$ and based on $A$ we define the vector fields $S_k = A \, \nabla_{g_E} u_{k}$ for $1\leq k \leq m$. Furthermore $A$ can be decomposed as follows
\begin{equation*}
    \Tilde{A}=\text{det}(A)^{-\frac{1}{2}} A \quad \text{with} \quad \text{det}(\Tilde{A})=1.
\end{equation*}
For the reconstruction procedure $S=(S_1,S_2)$ is orthonormalized into an $SO(2)$-valued frame $R=(R_1,R_2)$, by finding $T$ such that $R=S T^T$. The transfer matrix $T$ gives rise to the four vector fields $V_{bc}$ and thence $V_{bc}^{a}$:
\begin{equation*}
    V_{bc}=\sum_{k=1}^{2}\nabla_{g_E} (T_{bk}) T^{kc}, \quad 1 \leq b,c \leq 2, \quad V_{bc}^{a} = \frac{1}{2}(V_{bc}-V_{cb}),
\end{equation*}
where $T_{bk}$ and $T^{kc}$ denote the entries in $T$ and $T^{-1}$ respectively. As $R$ is a rotation matrix, it can be parameterized by a function $\theta$ as $R=\left(\begin{smallmatrix}\cos \theta & -\sin \theta\\
\sin \theta & \cos \theta\end{smallmatrix}\right)$. For the reconstruction procedure to work, we need the solutions $u_i$ that determine the power density matrix $H^E$ to satisfy the conditions \eqref{cond11E}-\eqref{cond12E}. Both conditions are directly motivated by the reconstruction formulas: The first condition guarantees invertibility of the submatrices $H^{E,(1)}$ and $H^{E,(2)}$ which are defined by:
\begin{equation*}
    H^{E,(1)}=\begin{bmatrix}
        H_{11}^E & H_{12}^E\\ H_{12}^E & H_{22}^E
    \end{bmatrix}, \quad \text{and} \quad H^{E,(2)}=\begin{bmatrix}
        H_{33}^E & H_{34}^E\\ H_{34}^E & H_{44}^E
    \end{bmatrix}.
\end{equation*}
The second condition ensures that a linear system of equations on the form $\Tilde{A}^2 X_f = Y_f$ never has the zero solution $X_f=0$ (this system of equations is spelled out explicitly in \eqref{xfyf}) and this condition is discussed further below.
The reconstruction procedure is then based on two equations that are corresponding to one pair of solutions $(u_i,u_j)$ to \eqref{eqPDENE} with $(i,j)=(1,2)$ or $(i,j)=(3,4)$ respectively. For simplicity we state the equations \eqref{eq8g} and \eqref{eq10g} corresponding to the pair $(u_1,u_2)$. Using  that $\Div_{g_E} (J A^{-1}S_k)$ vanishes on $M$ for $k=1,2$, and with $J$ defined as $J=\left(\begin{smallmatrix}0 & -1\\
1 & 0\end{smallmatrix}\right)$, one can derive the first equation:
\begin{equation}\label{eq8g}
    \nabla_{g_E} \log(\text{det} \,A)=D + \lp\nabla_{g_E} \lp H^{E,(1)}\rp^{qp} \cdot \Tilde{A}\, S_p\rp \Tilde{A}^{-1} \, S_q,
\end{equation}
with $D=\frac{1}{2}\nabla_{g_E} \log \lp \text{det}H^{E,(1)} \rp$ and where $\lp H^{E,(1)}\rp^{qp}$ denotes entries in $\lp H^{E,(1)}\rp^{-1}$. By writing the Lie bracket $[\Tilde{A}R_2, \Tilde{A}R_1]$ in two different ways one can obtain the second equation: 
\begin{equation}\label{eq10g}
    \tilde{A}^2 \nabla_{g_E} \theta + [\tilde{A}_2,\tilde{A}_1] =   \tilde{A}^2 V_{12}^{a} - \frac{1}{2}J D.
\end{equation}
Here $[\tilde{A}_2,\tilde{A}_1]$ denotes the Lie bracket between columns of $\tilde{A}$. The main challenge in the reconstruction procedure is to reconstruct $\Tilde{A}$ from equation \eqref{eq10g}, as the equation both depends on $\Tilde{A}$ and the unknown function $\theta$. For this purpose one needs two pairs of boundary conditions that both give rise to solutions that satisfy equation \eqref{eq10g}:
\begin{equation}
    \widetilde{A}^2 \nabla_{g_E} \theta_{k} + [\widetilde{A}_2, \widetilde{A}_1] = \widetilde{A}^2 V_{12}^{a(k)} - \frac{1}{2}J D^{(k)}.\label{eq102}
\end{equation}
Here $(k)$ in $\theta_k$, $V_{12}^{a(k)}$ and $D^{(k)}$ indicates the respective pair of solutions that gives rise to these functions and vector fields for $k=1,2$. Subtracting equation \eqref{eq102} with $k=1$ from the same equation with $k=2$ yields
\begin{equation}
    \Tilde{A}^2 \left( \nabla_{g_E}(\theta_2-\theta_1) - V_{12}^{a(2)} + V_{12}^{a(1)}\right)=-\frac{1}{2}J (D^{(2)}-D^{(1)}).\label{eq10sub}
\end{equation}
This is an algebraic system of equations on the form
\begin{equation*}
    \Tilde{A}^2 X_f = Y_f,
\end{equation*}
with
\begin{equation}\label{xfyf}
    X_f= \nabla_{g_E}(\theta_2-\theta_1) - V_{12}^{a(2)} + V_{12}^{a(1)}, \quad \text{and} \quad Y_f = -\frac{1}{2}J (D^{(2)}-D^{(1)}).
\end{equation}
Note that it is possible to express the functions $\cos(\theta_2-\theta_1)$ and $\sin(\theta_2-\theta_1)$, and hence also $\nabla_{g_E}(\theta_2-\theta_1)$ appearing in the term $X_f$ by the data, when taking inner products between columns of the two $R$-matrices $R^{(1)}$ and $R^{(2)}$:
\begin{align}
\begin{aligned}
        \cos(\theta_2-\theta_1)&=R_1^{(1)} \cdot R_1^{(2)} = \sum_{i,j=1}^2 T_{1i}^{(1)} T_{1j}^{(2)} H_{i(2+j)}^E\\
        \sin(\theta_2-\theta_1)&=R_2^{(1)} \cdot R_1^{(2)} = \sum_{i,j=1}^2 T_{2i}^{(1)} T_{1j}^{(2)} H_{i(2+j)}^E.\\
        \end{aligned}\label{BalMonardtrick}
\end{align}
By the chain rule it follows that
\begin{equation*}
    \nabla_{g_E} (\theta_2-\theta_1) = \cos(\theta_2-\theta_1) \nabla_{g_E} \sin(\theta_2-\theta_1) - \sin(\theta_2-\theta_1) \nabla_{g_E} \cos(\theta_2-\theta_1)
\end{equation*}
so the gradient $\nabla_{g_E} (\theta_2-\theta_1)$ is solely determined by entries of the $T^{(i)}$-matrices and the $4\times 4$ matrix $H^E$ that contains the power density data.

\subsubsection{Equations for the reconstruction procedure}
For the first step in the reconstruction procedure one reconstructs $\widetilde{A}$ from equation \eqref{eq10sub} and therefore needs data corresponding to $m=4$ boundary conditions. By the above analysis all quantities apart from $\widetilde{A}$ depend solely on the data. Therefore, solving this equation for $\widetilde{A}$ boils down to solving a linear system of equations.\par
For the second step in the reconstruction procedure one reconstructs the angle $\theta$ to be able to determine the vector fields $S_i=A \nabla_{g_E} u_{i}$ from the entries of $H_{ij}^E=\gamma \nabla_{g_E} u_{i} \cdot \nabla_{g_E} u_{j}$. For this one needs data corresponding to $m=2$ measurements. $\theta$ is reconstructed by solving the following gradient equation which is deduced from equation \eqref{eq10g}:
\begin{equation}\label{eq:thetagrad}
    \nabla_{g_E} \theta = F,
\end{equation}
with 
\begin{equation*}
    F = V_{12}^{a} - \widetilde{A}^{-2}\lp \frac{1}{2}J D + [\widetilde{A}_2,\widetilde{A}_1]\rp.
\end{equation*}
Once $\theta$ is known at at least one point on the boundary one can integrate $F$ along curves originating from that point to obtain $\theta$ throughout the whole domain. Alternatively, when assuming that $\theta$ is known along the whole boundary one can apply the divergence operator to \eqref{eq:thetagrad} and solve the following Poisson equation with Dirichlet boundary condition:
\begin{equation}\label{eq:thetaPois}
    \begin{cases}
     \Delta_{g_E} \theta = \nabla_{g_E} \cdot F & \text{in }M,\\
     \theta = \theta_{\text{true}} & \text{on }\partial M.
    \end{cases}
\end{equation}
For the third step in the reconstruction procedure one reconstructs the determinant of $A$ from equation \eqref{eq8g} requiring data from $m=2$ measurements. Equation \eqref{eq8g} can be simplified further to be on the following form:
\begin{equation*}
    \nabla_{g_E} \log(\det A) = I,
\end{equation*}
with 
\begin{align*}
    I &= \cos(2\theta) K + \sin(2\theta) J K,\\
    K &= U \tilde{A} (V_{11} - V_{22}) + J U \tilde{A} (V_{12}+V_{21}), \quad \text{and} \quad U = \begin{bmatrix} 1 & 0\\ 0 & -1 \end{bmatrix}.
\end{align*}
Similarly as for $\theta$, one needs to solve a gradient equation to obtain $\det(A)$ and has the possibility of either integrating along curves or solving a Poisson equation, assuming knowledge of $\gamma$ in one point or along the whole boundary respectively. We assume knowledge of $\gamma$ along the whole boundary and solve the following Poisson problem with Dirichlet condition:
\begin{equation}\label{eq:APois}
    \begin{cases}
     \Delta_{g_E} \log (\det A) = \nabla_{g_E} \cdot I & \text{in }M,\\
     \log(\det A) = \log(\det A_{\text{true}}) & \text{on }\partial M.
    \end{cases}
\end{equation}
These steps yield the reconstruction procedure outlined in Algorithm \ref{algo:Euclidean}.

\begin{algorithm}[H]
Ensure that $H^E$ satisfies \eqref{cond11E} and \eqref{cond12E}.
\begin{enumerate}
    \item Reconstruct $\Tilde{A}$ by solving equation \eqref{eq10sub} and using data from the full power density matrix $H^E$
    \item Reconstruct $\theta$ by solving the boundary value problem \eqref{eq:thetaPois} and using data from the submatrix $H^{E,(1)}$
    \item Reconstruct $(\det A)$ by solving the boundary value problem \eqref{eq:APois} and using data from the submatrix $H^{E,(1)}$
\end{enumerate}
 \caption{Euclidean reconstruction procedure in $(M,g_E)$}\label{algo:Euclidean}
\end{algorithm}\textbf{}

\subsubsection{Choice of the transfer matrix $T$ and knowledge about $\theta$ used for numerical experiments}
In the following we describe the transfer matrix $T$ determined by $R=ST^T$ corresponding one pair of solutions $(u_i,u_j)$ where $(i,j)=(1,2)$ or $(i,j)=(3,4)$. For simplicity we state the results corresponding to the pair $(u_1,u_2)$ with corresponding submatrix $H^{E,(1)}$. The matrix $T$ with corresponding rotation matrix $R$ is uniquely defined up to a rotation. In theory, the choice of $T$ will not influence the reconstruction procedure, as every choice of $T$ with corresponding $\theta$ will work to extract the vector fields $S_k=A \nabla_{g_E} u_{k}$ from the entries of $H^{E,(1)}$. However, numerically a simple choice of $T$ can be an advantage. For this reason, we choose Gram-Schmidt orthonormalisation to obtain the following $T$, as in this case the vector fields $V_{bc}$ have the simplified form as in \eqref{eq:Vexp}:
\begin{equation}\label{eq:Tdef}    
    T^E = \begin{bmatrix} (H_{11}^E)^{-\frac{1}{2}} & 0 \\ -H_{12}^E (H_{11}^E)^{-\frac{1}{2}} d^{-1} &(H_{11}^E)^{\frac{1}{2}} d^{-1} \end{bmatrix},
\end{equation} 
with $d=(H_{11}^E H_{22}^E -(H_{12}^E)^2)^{\frac{1}{2}}$. By the Jacobian condition~\eqref{cond11E}, $H_{11}^E>0$ and thus $T$ is well-defined. For this choice of $T$ the function $\theta$ is given by the angle between $A\nabla_{g_E} u_{1}$ and the $x_1$-axis, as in this case the first column of $R$ simplifies to
\begin{equation*}
    R_1=T_{11} S_1 + T_{12} S_2 = \dfrac{A \nabla_{g_E} u_{1}}{\vert A\nabla_{g_E} u_{1}\vert_{g_E}},
\end{equation*}
so that
\begin{equation}\label{eq:thetaAng}
    \theta=\text{arg}(A \nabla_{g_E} u_{1}).
\end{equation}
In addition, the vector fields $V_{bc}$ can be written explicitly in terms of $H^{E,(1)}$:
\begin{align}\label{eq:Vexp}
\begin{aligned}
    V_{11}&= \nabla_{g_E} \log (H_{11}^E)^{-\frac{1}{2}}, & V_{12}&=0,\\ V_{21}&= -\frac{H_{11}^E}{d} \nabla_{g_E} \left( \frac{H_{12}^E}{H_{11}^E}\right), & V_{22}&=\nabla_{g_E} \log \left( \frac{(H_{11}^E)^{\frac{1}{2}}}{d}\right).
    \end{aligned}
\end{align}
By the maximum principle \cite{Gilbarg:Trudinger:2001} $u_{1}$ achieves its minimum over $M$ at $x_m$. Therefore, at this point the gradient points in the direction of highest increase of $u_{1}$  corresponding to $-\nu$, the inward normal vector (by condition \eqref{cond11E} $\nabla_{g_E} u_{1}$ can never be the zero vector; this implies $\vert A \nabla_{g_E} u_1 \vert_{g_E}>0$). Hence,
\begin{equation*}
    \frac{A\nabla_{g_E} u_{1}}{\vert A\nabla_{g_E}u_{1} \vert_{g_E}}(x_m)= -\nu(x_m) \quad \Rightarrow \quad  A \nabla_{g_E} u_{1}(x_m) = -\vert A\nabla_{g_E} u_{1} \vert_{g_E} \, \nu (x_m),
\end{equation*}
so that $\theta$ is known at $x_m$. Since the unit outward normal $\nu$ is known at the boundary, $\theta$ is solely determined by the direction of $A\nabla_{g_E} u_{1}$. Knowledge of $\theta$ along the whole boundary is related to knowledge of the Neumann data $\gamma \nabla_{g_E} u_{1} \cdot \nu$. This follows from the fact that $A \nabla_{g_E} u_{1}$ can be decomposed into two parts with contribution from the unit normal $\nu$ and the tangent vector $t=J \nu$: 
\begin{equation*}
    A \nabla_{g_E} u_{1} = (A \nabla_{g_E} u_{1} \cdot \nu)\nu + (A \nabla_{g_E} u_{1} \cdot t)t.
\end{equation*}
As $u_{1}\vert_{\partial M}$ is known by the Dirichlet condition and $\gamma$ and thus $A$ is assumed known at the boundary, complete knowledge of $\theta$ along the boundary follows when the Neumann data $\gamma \nabla_{g_E} u_{1} \cdot \nu$ is known.

\subsubsection{Simplification of the equations when using information from $m=3$ boundary functions}
In numerical simulations $m=3$ boundary functions has proven sufficient for reconstructing anisotropic conductivities. In accordance with \cite{BalMonard} we therefore use three boundary conditions in pairs $(f_1,f_2)$ and $(f_3,f_4)$ with $f_3=f_2$ for our numerical simulations. In this section we highlight how the $T^{(k)}$-matrices, $V_{bc}^{(k)}$-vector fields and the corresponding vector fields $X_f$ and $Y_f$ in \eqref{xfyf} simplify in this situation.\\\\
Note that as we only use three boundary conditions, the power density matrix $H^E$ has recurring elements and this implies that the expressions for the transfer matrices $T^{(1)}$ and $T^{(2)}$ used for the Euclidean reconstruction procedure are simplified. $T^{(1)}$ and $T^{(2)}$ correspond to the submatrices $H^{E,(1)}$ and $H^{E,(2)}$ of the Euclidean power density matrix $H^E$. Using the representation of $T$ as in \eqref{eq:Tdef} gives:
\begin{align*}
    T^{(1)}&= \begin{bmatrix}
        (H_{11}^E)^{-\frac{1}{2}} & 0\\ -H_{12}^E (H_{11}^E)^{-\frac{1}{2}} d_1^{-1} & (H_{11}^E)^{\frac{1}{2}} d_1^{-1}
    \end{bmatrix},
    \intertext{and}
    T^{(2)}&= \begin{bmatrix}
        (H_{22}^E)^{-\frac{1}{2}} & 0\\ -H_{24}^E (H_{22}^E)^{-\frac{1}{2}} d_2^{-1} & (H_{22}^E)^{\frac{1}{2}} d_2^{-1}
    \end{bmatrix},
\end{align*}
with $d_1 = (H_{11}^E H_{22}^E -(H_{12}^E)^2)^{\frac{1}{2}}$ and $d_2=(H_{22}^E H_{44}^E-(H_{24}^E)^2)^{\frac{1}{2}}$. Furthermore, the vector fields $V_{bc}^{(k)}$ are on the form (following the formulas in \eqref{eq:Vexp}):
\begin{align*}
\begin{aligned}
    V_{11}^{(1)}&= \nabla_{g_E} \log (H_{11}^E)^{-\frac{1}{2}}, & V_{11}^{(2)}&= \nabla_{g_E} \log (H_{22}^E)^{-\frac{1}{2}},\\
    V_{12}^{(1)}&=0, & V_{12}^{(2)}&=0,\\ 
    V_{21}^{(1)}&= -\frac{H_{11}^E}{d_1} \nabla_{g_E} \left( \frac{H_{12}^E}{H_{11}^E}\right), & V_{21}^{(2)}&= -\frac{H_{22}^E}{d_2} \nabla_{g_E} \left( \frac{H_{24}^E}{H_{22}^E}\right),\\
    V_{22}^{(1)}&=\nabla_{g_E} \log \left( \frac{(H_{11}^E)^{\frac{1}{2}}}{d_1}\right), & V_{22}^{(2)}&=\nabla_{g_E} \log \left( \frac{(H_{22}^E)^{\frac{1}{2}}}{d_2}\right).
    \end{aligned}
\end{align*}
Using these expressions for the $T$-matrices and the vector fields $V_{bc}$ one can arrive at the simplified forms for $X_f$ and $Y_f$ in equation \eqref{xfyf} used for reconstruction of $\tilde{\gamma}$ in the Euclidean procedure:
\begin{align*}
    X_f&=-\frac{H_{22}^E}{2} \lp \frac{1}{d_1} \nabla_{g_E} \lp \frac{H_{12}^E}{H_{22}^E}\rp + \frac{1}{d_2} \nabla_{g_E} \lp \frac{H_{24}^E}{H_{22}^E}\rp\rp,
    \intertext{and}
    Y_f&=-\frac{1}{2}J \nabla_{g_E} \lp \log d_2 - \log d_1\rp.
\end{align*}
Note that the element $H_{14}^E$ does not appear in the formula and is thus not needed for the reconstruction procedure.
\subsection{Reconstruction procedure for conformally parametrized manifolds}
The reconstruction procedure follows directly from section \ref{secMainNew} and the constructive Euclidean reconstruction procedure in Algorithm \ref{algo:Euclidean}. It is outlined in Algorithm \ref{algoConf}.

\begin{algorithm}[H]
Choose a set of boundary conditions $(f_{1},f_{2},f_{3},f_{4})$ so that $H$ satisfies \eqref{cond11M} and \eqref{cond12M}.
\begin{enumerate}
        \item Transform the power density data to the Euclidean domain $(M,g_E)$: $$H^E(x) = \rho^2 \, H(x)$$
        \item As $H$ satisfies the conditions \eqref{cond11M}-\eqref{cond12M}, the corresponding $H^E$ satisfies the conditions \eqref{cond11E}-\eqref{cond12E}. Use the Euclidean procedure in Algorithm \ref{algo:Euclidean} to reconstruct $\gamma(x)$ from $H^E(x)$ in $(M,g_E)$
\end{enumerate}
 \caption{Reconstruction procedure for conformal parametrizations}\label{algoConf}
\end{algorithm}

\subsection{Reconstruction procedure for non-conformally parametrized manifolds}
\label{secRecProcNonConf}
The reconstruction procedure follows then from the analysis in section \ref{secChoice} and the reconstruction procedure for conformally parametrized manifolds in Algorithm \ref{algoConf}. The idea for this setting is that in practice it is not straightforward to find a conformal parametrization of the manifold straight away. So in order to use a constructive approach for finding this parametrization one needs to employ \cite{gu2002a,gu2004a,schoen1997a}. This requires to make a choice of coordinates in the first place that most likely will give a non-conformal parametrization of the manifold. The procedure is then outlined in Algorithm \ref{algoM}.

\begin{algorithm}[H]
Choose a set of boundary conditions $(f_{N,1},f_{N,2},f_{N,3},f_{N,4})$ so that $H^N$ satisfies \eqref{cond11N} and \eqref{cond12N}.
\begin{enumerate}
        \item Find the conformal diffeomorphism $\psi$ from $(N,g_N)$ to $(M,g_M)$
        \item Compute $\rho$ from 
        $$(D\psi^{-1})^t(x) G_N(\psi^{-1}(x)) D\psi^{-1}(x)=\rho^2(x) \cdot G_E$$
        and express $H^N$ in the conformal coordinates: $$\quad H^M(x)=H^N(y)$$
        \item Transform the power density data to the Euclidean domain $(M,g_E)$: $$H^E(x) = \rho^2 \, H^M(x)$$
        \item As $H^N$ satisfies the conditions \eqref{cond11M}-\eqref{cond12M}, the corresponding $H^E$ satisfies the conditions \eqref{cond11E}-\eqref{cond12E}. Use the Euclidean procedure in Algorithm \ref{algo:Euclidean} to reconstruct $\gamma_M(x)$ from $H^E(x)$ in $(M,g_E)$
\end{enumerate}
 \caption{Reconstruction procedure for other parameterizations.}\label{algoM}
\end{algorithm}

\section{Representing anisotropic conductivities as ellipse fields}\label{sec:ellipses}
Since the conductivity at each point is a linear map of the tangent plane into itself, it is difficult to visualize $\gamma$ in a coordinate invariant way even on a $2$-dimensional surface. For example, if we just display the three components of the corresponding matrix valued function on the surface, then this will yield different results for different parametrizations. 

Instead, at each point $p$ on the surface we compute the corresponding action of $\gamma$ on the unit vectors in the tangent plane at $p$, i.e. we compute and display the ellipse  
\begin{equation*}
    E(p)=\left\{ \gamma(X) \in T_p M \vert \, g(X, X) =1\right\}.
\end{equation*}
    
This gives a simple visualization of the $(1,1)$ tensor field $\gamma$. Indeed, suppose $\gamma$ has the two eigenvectors $a_{1}$ and $a_{2}$ with corresponding eigenvalues $\lambda_{1}$ and $\lambda_{2}$. The eigenvalues are positive and the eigenvectors are $g$-orthogonal and all is represented by the ellipse

\begin{equation*}
 E(p)=\left\{ \lambda_{1}\cdot \frac{a_{1}}{\Vert a_{1} \Vert_{g}}\cdot \cos(\theta) + \lambda_{2}\cdot \frac{a_{2}}{\Vert a_{2} \Vert_{g}}\cdot \sin(\theta) \in T_p M \vert \, \theta \in \mathbb{S}^{1} \right\}.
\end{equation*}

\subsection{Comparing conductivities}
We need to be able to compare conductivities both pointwise and globally via a suitable measure of difference. At each point we can do that by comparing the corresponding representing ellipses:
Suppose that $\lambda_{1} \geq \lambda_{2}$ are the positive eigenvalues of $\gamma$ with corresponding $g$-orthogonal \emph{unit} eigenvectors $c_{1}$ and $c_{2}$. Suppose that $\mu_{1} \geq \mu_{2}$ are the positive eigenvalues of $\eta$ with corresponding $g$-orthogonal \emph{unit} eigenvectors $d_{1}$ and $d_{2}$.
We assume further that the two pairs of eigenvectors in the given order define the same orientation in the tangent plane.
Then there is a unique minimal rotation by some angle $\phi$ of $c_{1}$ and $c_{2}$ onto $d_{1}$ and $d_{2}$ followed by unique scalings in the $d_{1}$ and $d_{2}$ directions so that in total
$\lambda_{1}\cdot c_{1}$ and $\lambda_{2} \cdot c_{2}$ are mapped onto $\mu_{1}\cdot d_{1}$ and $\mu_{2} \cdot d_{2}$, respectively. There are now different choices of measures of the
energy of this map, i.e. different choices of invariant measures of the difference between the two conductivities $\gamma$ and $\eta$. The SVD decomposition of the total map gives immediately the two singular values
\begin{equation}
    \sigma_{1} = \frac{\mu_{1}}{\lambda_{1}}\quad , \quad   \sigma_{2} = \frac{\mu_{2}}{\lambda_{2}}\quad .
\end{equation}

One possible general measure of difference, which is also 
invariant under interchange of $\gamma$ and $\eta$ (the contribution of the singular values of the deformation are the same as the contribution from the inverse deformation), is the following:

\begin{equation}\label{ellipseDist}
m^{2}(\gamma, \eta) = \kappa \cdot \phi^{2} + \left( \frac{\lambda_{1}}{\mu_{1}}- \frac{\mu_{1}}{\lambda_{1}}\right)^2 + \left( \frac{\lambda_{2}}{\mu_{2}} - \frac{\mu_{2}}{\lambda_{2}}\right)^{2} \quad,
\end{equation}
where $\kappa$ is a constant (of choice), which determines the relative weight of the rotation part of the mapping. An alternative measure of difference between conductivities can be obtained via a distance function on the set of symmetric positive definite matrices as presented in e.g. \cite{Moakher2011}.

\section{A computational study}\label{secNumExample}
In the following  the manifold $(M,g)$ is a domain on the catenoid as shown in figure \ref{fig:cat}. We aim at reconstructing the conductivity $\gamma$ visualized by its corresponding ellipse field in figure \ref{fig:trueEllipsefield}. For the numerical reconstruction procedure we consider the following conformal parametrization $r$ of $(M,g=\rho^2 g_E)$ with the corresponding conformal factor $\rho$:
\begin{align}
    r(x^1,x^2)&=\begin{pmatrix}
    \cosh\lp\log \lp \sqrt{(x^1)^2+(x^2)^2}\rp \rp\cos(\arg(x^1+i x^2))\\
    \cosh \lp \log \lp \sqrt{(x^1)^2+(x^2)^2}\rp \rp\sin(\arg(x^1+i x^2))\\
     \log \lp \sqrt{(x^1)^2+(x^2)^2}\rp \\
    \end{pmatrix} \label{eqConfParam}\\
    \rho&=\rho(x^1,x^2) =\frac{\cosh\lp \frac{1}{2}\log((x^1)^2+(x^2)^2)\rp}{\sqrt{(x^1)^2+(x^2)^2}}. \label{eqConfFactor}
\end{align}
The corresponding parameter domain and the square of the conformal factor are illustrated in figure \ref{fig:rho}.
\begin{figure}
    \centering
        \includegraphics[width=0.4\linewidth]{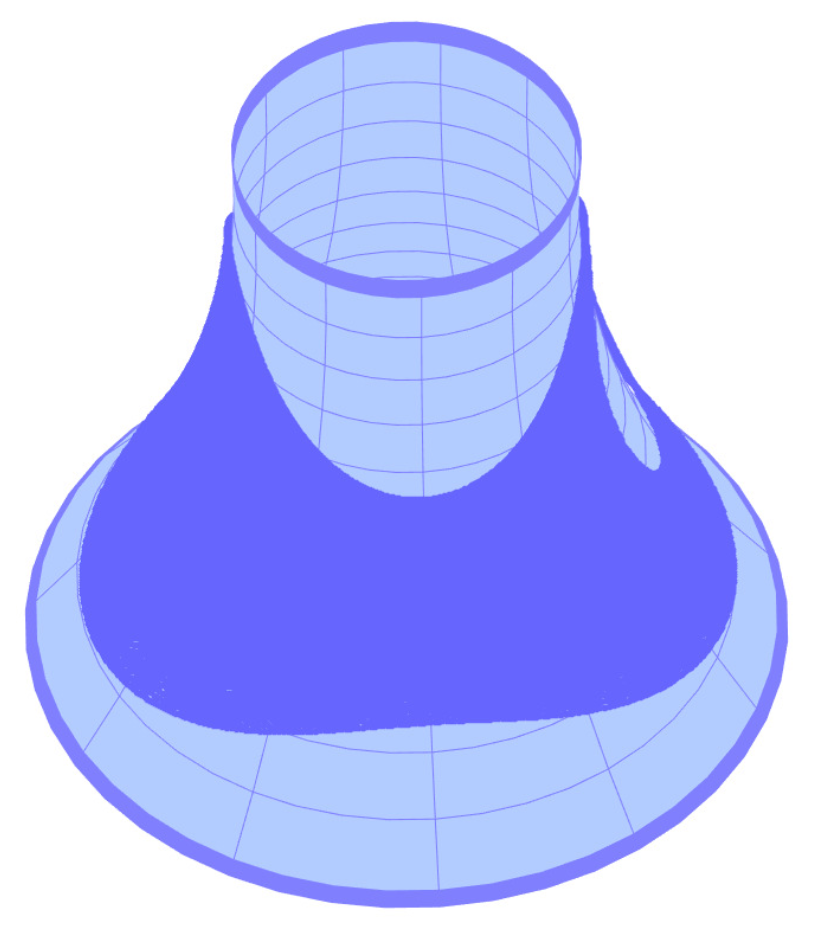}
        \caption{The manifold $(M,g)$ represented by a catenoid.}
    \label{fig:cat}
\end{figure}

\begin{figure}
    \centering
        \includegraphics[width=0.7\linewidth]{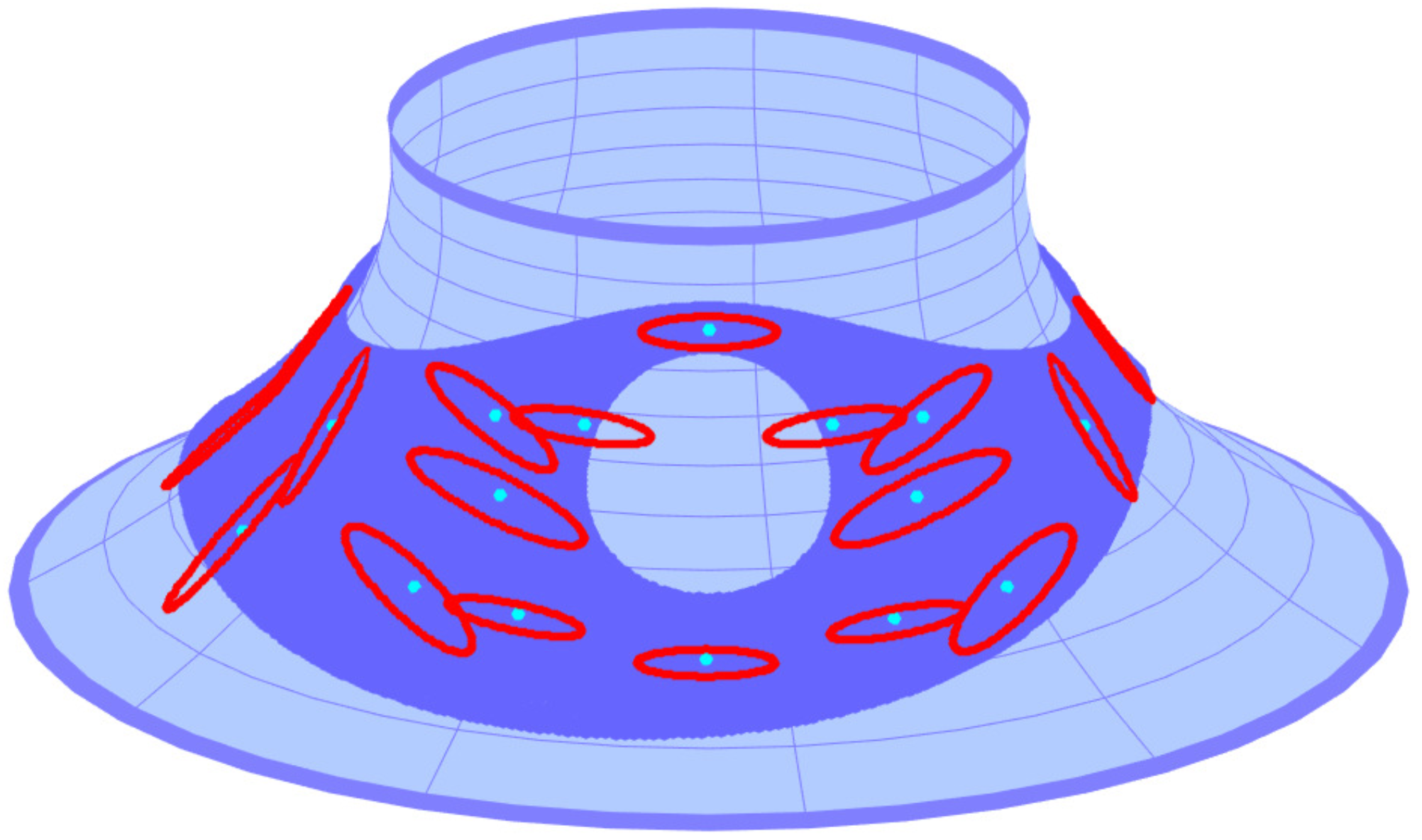}
        \caption{The true conductivity $\gamma$ visualized by its ellipse field on $(M,g)$.}
    \label{fig:trueEllipsefield}
\end{figure}
 
\begin{figure}
    \centering
        \includegraphics[width=\linewidth]{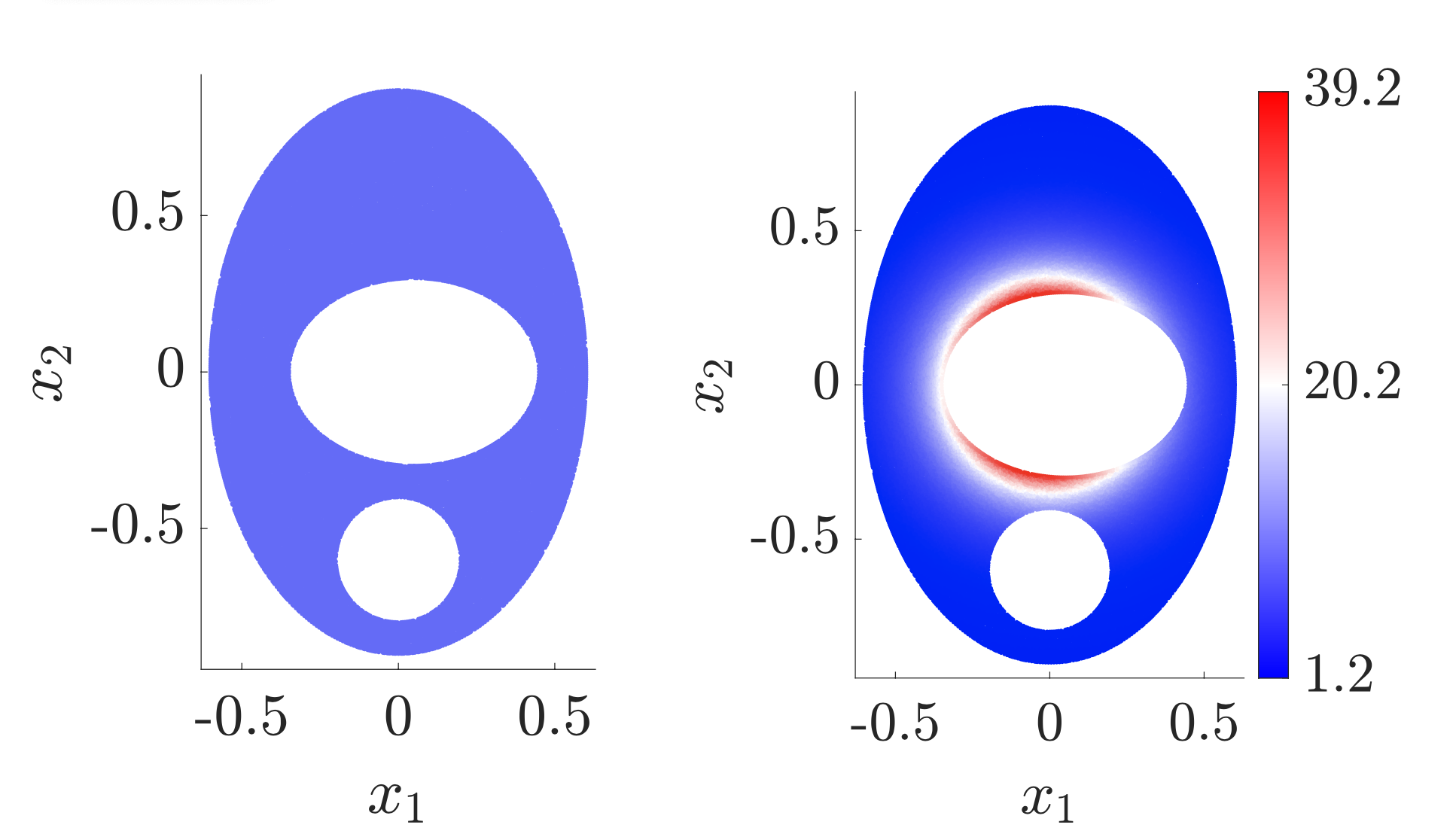}
        \caption{The parameter domain $M$ for the conformal parametrization of $(M,g)$ (left) and the square of the conformal factor, $\rho^2$, in $M$ (right).}
    \label{fig:rho}
\end{figure}
In the following we now discuss the implementation  of -- and numerical details for -- the reconstruction procedure of the conductivity on a compact subset of the catenoid.

\subsection{Implementation details}
The \textsc{Matlab} and \textsc{Python} code to generate the numerical example can be found on \textsc{GitLab} \cite{hjscGIT}

The Euclidean reconstruction procedure is implemented in \textsc{Python} and we use \textsc{FEniCS} \cite{fenics} to solve the PDEs. The power density data on $(M,g)$ is generated on a fine mesh with $N_1=94450$ nodes and we use a coarser mesh with $N_2=42013$ nodes to address the Euclidean reconstruction procedure. For both meshes, we use $\mathbb{P}_2$ elements. It is essential that we use second order basis functions in our numerical simulations; it was not possible with first order basis functions to compute approximate solutions so that condition \eqref{cond12M} and thus \eqref{cond12E} were satisfied. We note that condition \eqref{cond11M} can be met also with first order basis functions. We explain this phenomenon with the fact that \eqref{cond11M} does not involve any derivatives so that the computation of the derivative in the left hand side of \eqref{cond12M} requires the extra information by using second order basis functions.  
\label{secNumConf}



\subsection{Boundary functions} 
In accordance with \cite{BalMonard} we use only three boundary conditions to generate the power densities. These are simple polynomials in $x^{1}$ and $x^{2}$  given by  $$(f_{1},f_{2},f_{4})=(-x^2-0.1 (x^2)^2,x^1-x^2,0.2x^1x^2+x^2-0.1(x^1)^2)$$ (the third boundary condition is $f_{3}=f_{2}$.) According to our numerical computations with the chosen conductivity, the corresponding solutions $(u_{1},u_{2},u_3,u_{4})$ satisfy the conditions \eqref{cond11M}-\eqref{cond12M} and are used to construct the power density matrix $H$. 
The element $H_{44}$ is illustrated in figure \ref{fig:PowDen}.


\begin{figure}
    \centering
        \includegraphics[width=\linewidth]{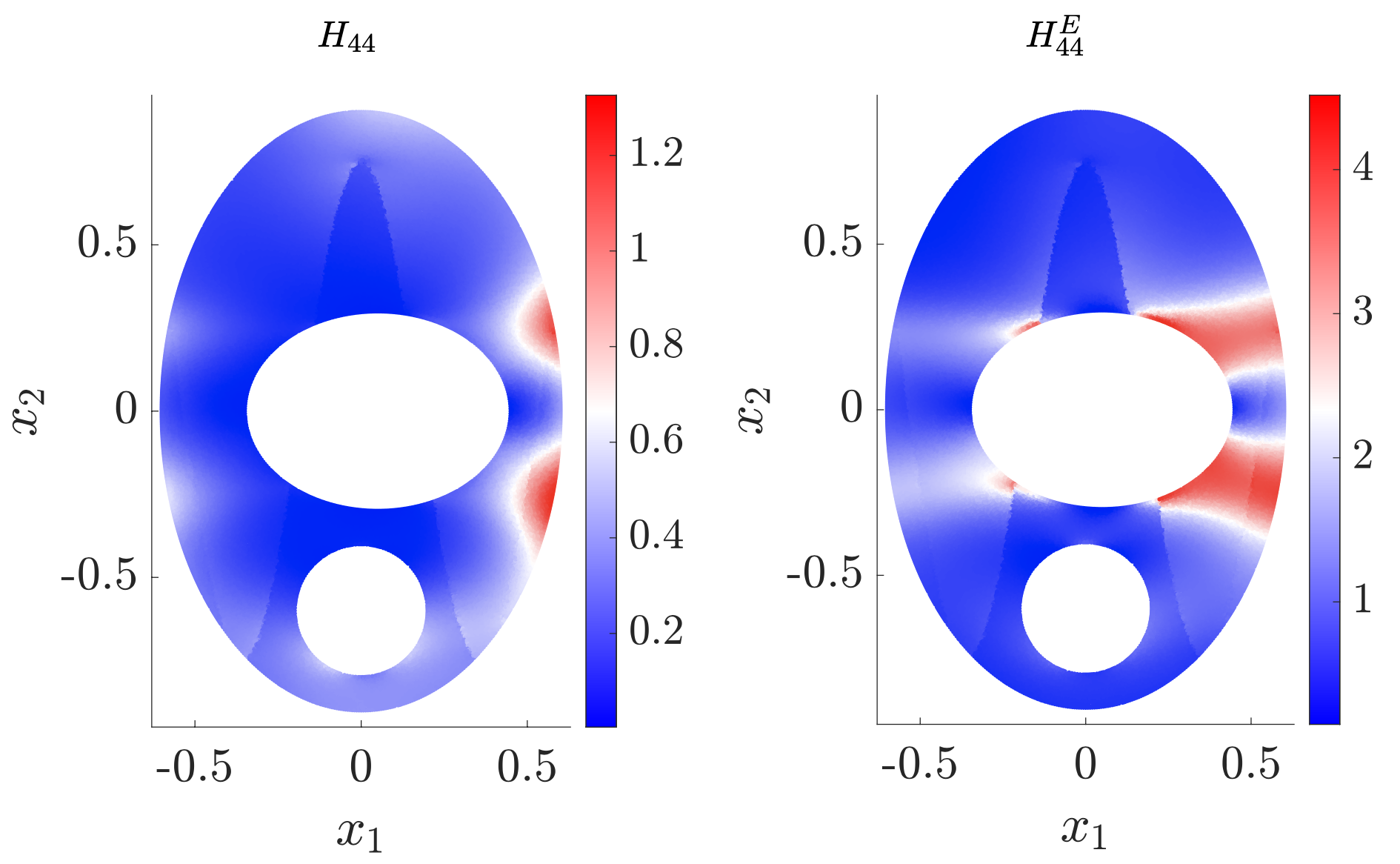}
        \caption{The transformation of the power density data from the manifold $(M,g)$ to the plane $(M,g_E)$ illustrated for the element $H_{44}$.}
    \label{fig:PowDen}
\end{figure}

In the following we go through the procedure in algorithm \ref{algoConf} in order to reconstruct $\gamma$ from the noise free $H$ and from a power density matrix perturbed by noise.

\subsubsection{Reconstruction from noise free data}

We first compute the power density data $H^E$ for the Euclidean domain $(M,g_E)$:
\begin{equation*}
    H^E(x^1,x^2)=\rho^2(x^1,x^2) \, H(x^1,x^2).
\end{equation*}
The transformation of the power density data from $H$ to $H^E$ is illustrated for the element $H_{44}$ in figure \ref{fig:PowDen}.

%
%
%
%

We parameterize $\gamma$ by three functions $\beta, \xi$ and $\zeta$:
\begin{equation}\label{gamNparam}
\gamma(\beta,\xi,\zeta)=\beta \underbrace{\begin{bmatrix}\xi & \zeta\\
\zeta & \frac{1+\zeta^2}{\xi}\end{bmatrix}}_{\text{{\Large $=\widetilde{\gamma}$}}} \; ,
\end{equation}
with $\beta = (\text{det }\gamma)^{\frac{1}{2}}$. The functions $\xi$ and $\zeta$ determine the normalized part of $\gamma$ denoted by $\widetilde{\gamma}$ and the function $\beta$ determines the determinant of $\gamma$. The true functions $\xi, \zeta$ and $\beta$ are given by the following expressions:
\begin{align*}
    \xi(x^1,x^2)&=2+\frac{3}{2} e^{-3\lp \lp x^1+\frac{3}{5}\rp^2+\lp x^2-\frac{5}{8}\rp^2 \rp}\\
    \zeta(x^1,x^2)&=\cos(2\arg(x^1+ix^2))\sin(2\arg(x^1+ix^2))\\
    \beta(x^1,x^2)&= \begin{cases}
    1+e^{-10(x^1)^2}+e^{-10(x^1-1)^2}+e^{-10(x^1+1)^2} & \text{for $\frac{3}{4}\cos\lp \frac{14\pi}{5}x^1\rp \leq x^2$}\\
    1 & \text{otherwise}.
    \end{cases}
\end{align*}
These choices are illustrated in the first row in figure \ref{fig:RecPlane}.
We now follow the Euclidean reconstruction procedure outlined in algorithm \ref{algo:Euclidean}. Here the first step consists of reconstructing $\widetilde{\gamma}$ and hence the functions $\xi$ and $\zeta$. The second step consists of reconstructing the angle $\theta$ to split the functionals $S_i = A \nabla_{g_E} u_{i}$ apart in the expression for the power densities $H_{ij}^E=A \nabla_{g_E} u_{i} \cdot A \nabla_{g_E} u_{j}$ (and for that only one submatrix of $H^E$ corresponding to $H^{E,(1)}$ is considered). The last step then consists of reconstructing the determinant of $\gamma$, hence the function $\beta$. The reconstructions are illustrated in the second row in figure \ref{fig:RecPlane}. The relative $L_2$-errors are given by 3.26\% ($\xi$), 7.62\% ($\zeta$) and 2.10\% ($\beta$). We note that the errors of $\xi$ and $\zeta$ are quite high relative to the fact that there is no noise in the data. The most difficult part for reconstruction is the piecewise constant cosine-curve appearing in $\beta$. Along this curve there appear artifacts in the reconstructions of all functions, but especially leading to the high errors for $\xi$ and $\zeta$. Furthermore,  artifacts appear in the reconstruction of $\xi$ and $\zeta$ around the points $(-0.46,0.31), (-0.24,-0.38), (0.38,-0.28)$ and $(0.31,-0.72)$. These artifacts are induced by the fact that the values on the left hand side in condition \eqref{cond12E} are very close to zero at these points. 

\begin{figure}[h!]
    \centering
    \includegraphics[width=\textwidth]{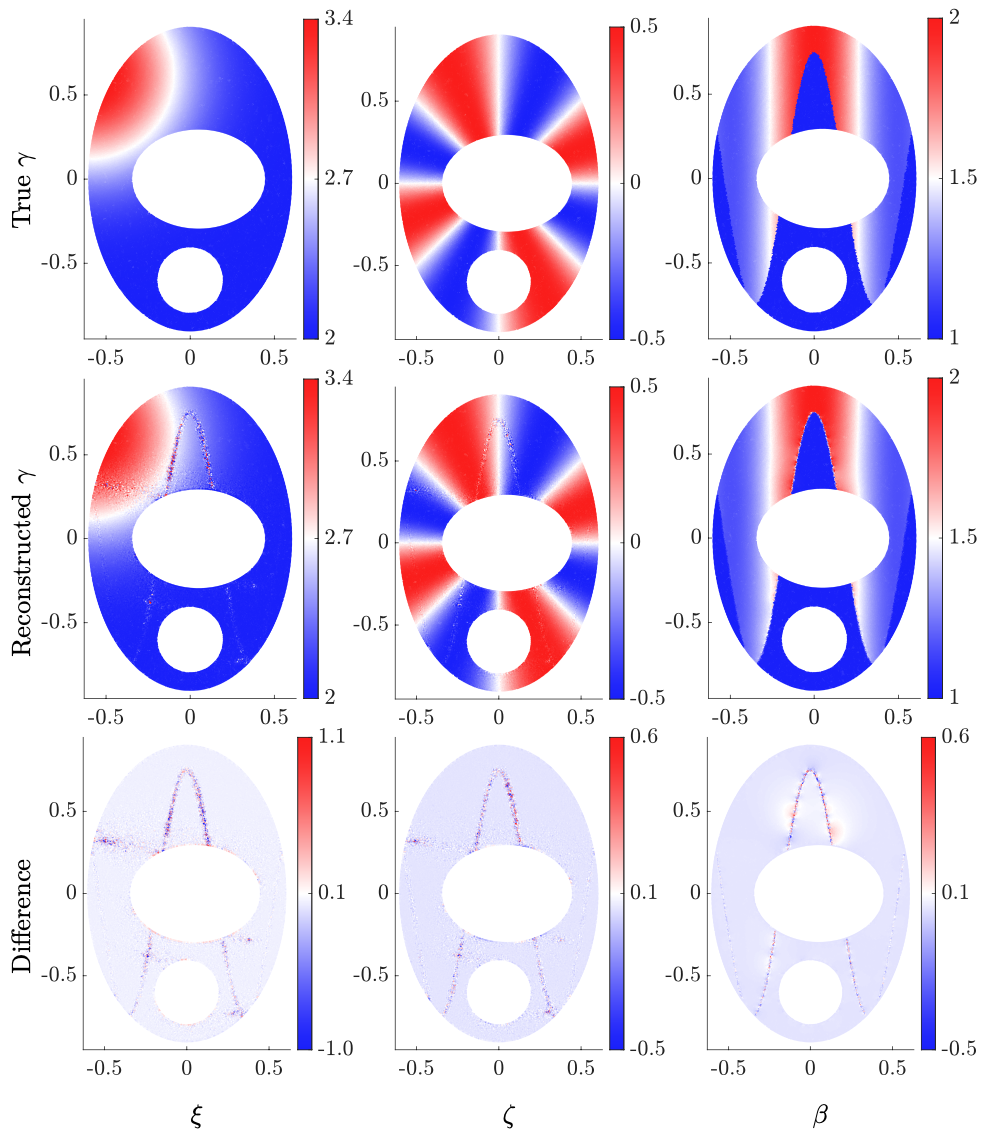}
    \caption{The true scalar functions $\xi$, $\zeta$ and $\beta$ determining the conductivity in $(M,g_E)$ (first row), their reconstructions (second row) and their difference (third row).}
    \label{fig:RecPlane}
\end{figure}

The ellipse field corresponding to $\gamma$ as discussed in section \ref{sec:ellipses} is illustrated (modulo a unifom scaling) in figure \ref{fig:RecCat}.


\begin{figure}
    \centering
    \includegraphics[width=0.7\linewidth]{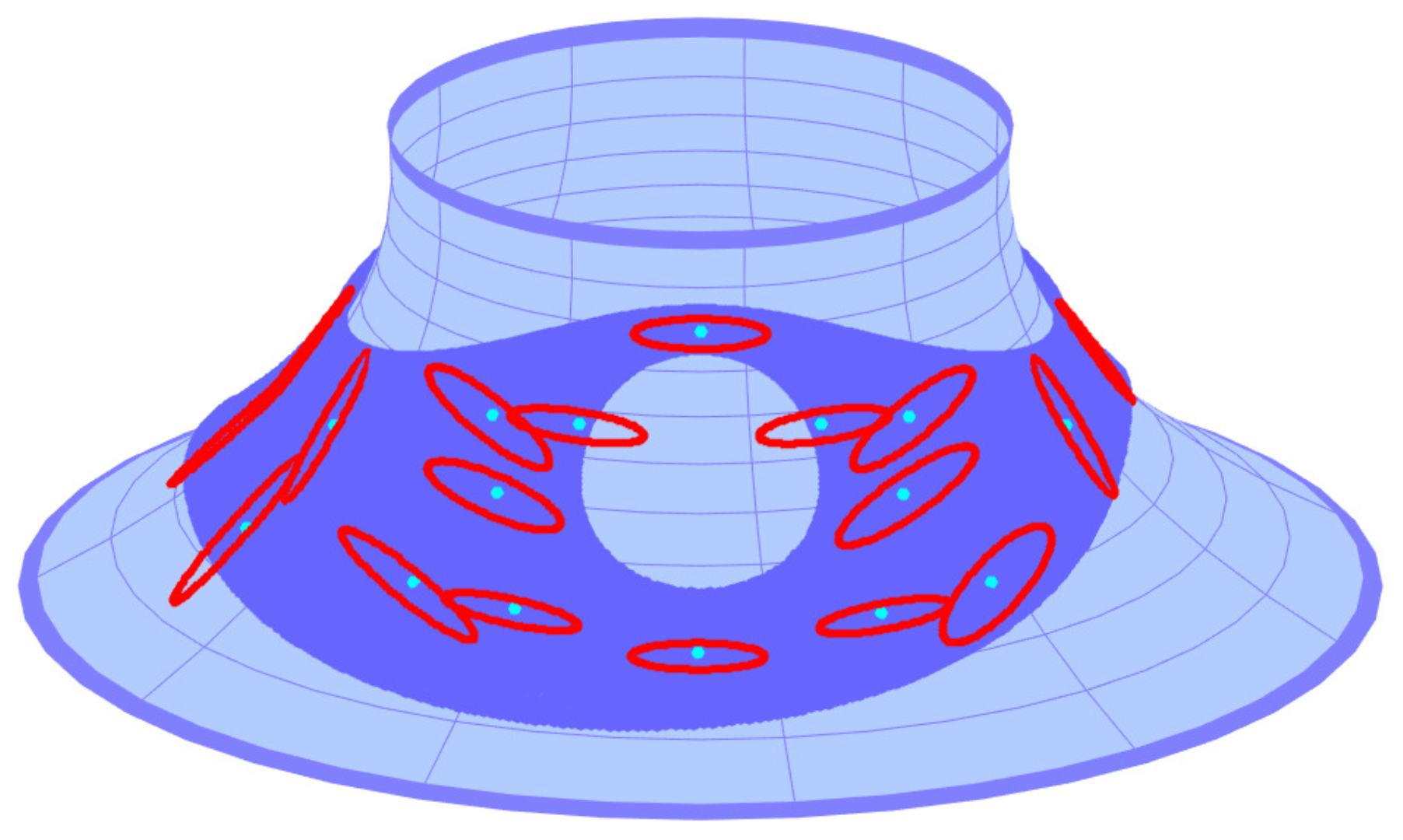}
    \caption{The reconstruction of $\gamma$ illustrated by ellipses in the tangent planes of $(M,g)$.}
    \label{fig:RecCat}
\end{figure}

To compare the reconstructed $\gamma$ with the true conductivity, $\gamma_{\text{true}}$, we use the difference measure in \eqref{ellipseDist}. We use the choice $\kappa=\frac{1}{4} \lp \lambda_1+\mu_1\rp^2$ and illustrate the relative difference $m^2(\gamma,\gamma_{\text{true}})/\kappa$ in figure \ref{fig:DiffRecCat}. Furthermore, we compute the accumulated relative difference:
\begin{equation*}
    \int_M \frac{m^2(\gamma,\gamma_{\text{true}})}{\kappa} \, \mathrm{d}V = 0.1306.
\end{equation*}
We observe that there are artifacts appearing along curves next to the circular hole. These correspond to the artifacts in the reconstruction of the planar functions, $\xi, \zeta$ and $\beta$, induced along the piecewise constant sine-curve appearing in $\beta$. Furthermore, there appear point clouds around the values, where condition \eqref{cond12M} and thus condition \eqref{cond12E} are close to being violated. 


\begin{figure}[h!]
    \centering
    \begin{minipage}[t]{\textwidth}
        \begin{minipage}[t]{\textwidth}
        \centering
        \includegraphics[width=0.6\linewidth]{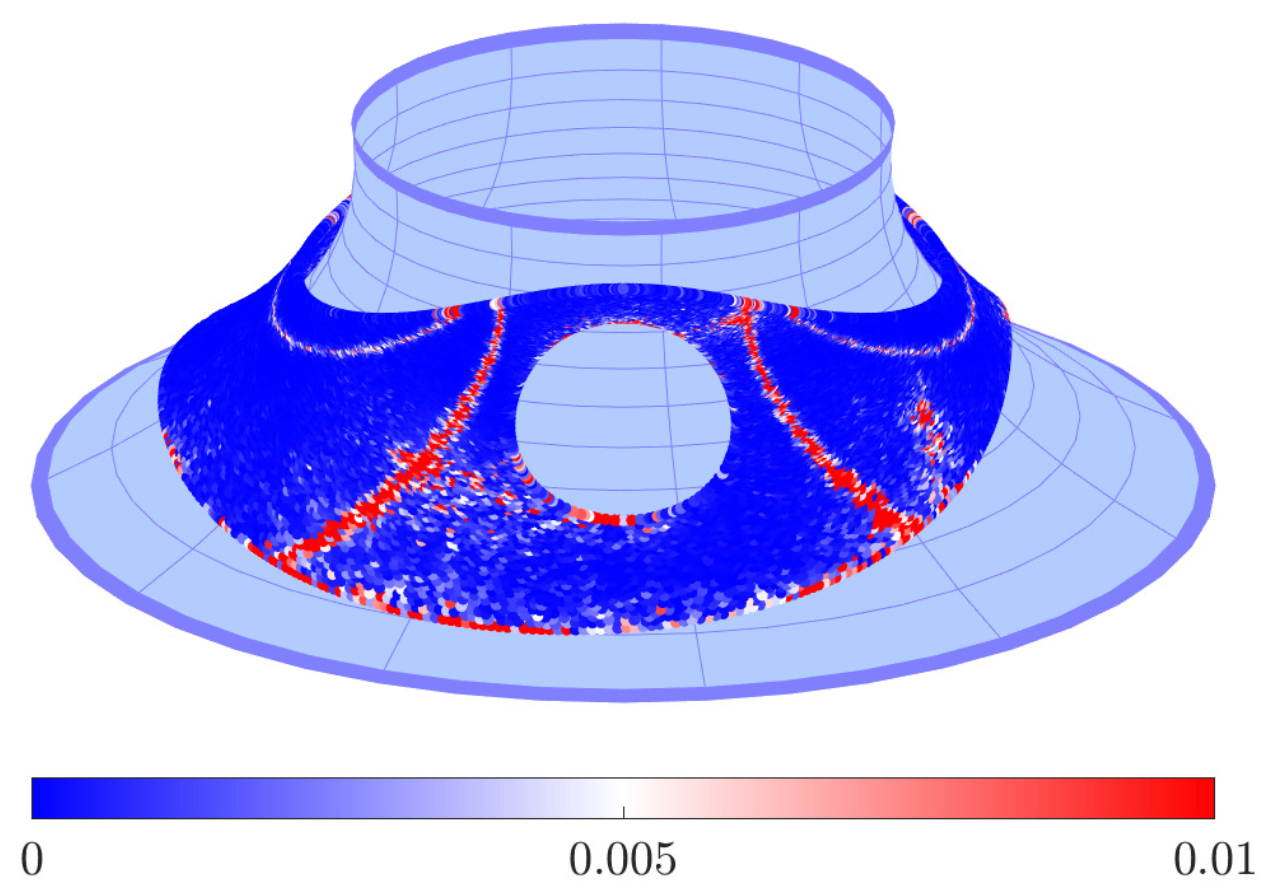}
    \end{minipage}
            \end{minipage}
    \caption{The coordinate invariant difference between the reconstructed conductivity, $\gamma$ and the true conductivity, $\gamma_{\text{true}}$, given by $m^2(\gamma,\gamma_{\text{true}})/\kappa$.}
    \label{fig:DiffRecCat}
    \end{figure}

\subsubsection{Reconstruction from noisy data}
We perturb the entries of the power density matrix $H$ on $(M,g)$ at each node with random noise:
\begin{equation*}
    \widetilde{H}_{ij} = H_{ij} + \frac{\alpha}{100}\frac{\norm{H_{ij}}_{L^2(M)}}{\norm{e_{ij}}_{L^2(M)}}e_{ij}
\end{equation*}
where $\alpha$ is the desired noise level and $e_{ij}$ are random perturbations collected as entries in the matrix $E.$ The norm is computed with respect to the metric $g=\rho^2 \, g_E$. The noise is generated in the discrete setting by drawing element wise Gaussian random variables from $\mathcal N(0,1)$ before normalising $E$. To maintain symmetry of $\widetilde{H}$ we compute $\frac{1}{2}(\widetilde{H}_{12}+\widetilde{H}_{21})$ for the off diagonal elements. The perturbation by noise is relative to the data on the manifold, but this is not the case for the Euclidean domain $(M,g_E)$. So after the transformation 
\begin{equation*}
    \widetilde{H}^E(x^1,x^2)=\rho^2(x^1,x^2) \, \widetilde{H}(x^1,x^2).
\end{equation*}
the noise is no longer Gaussian due to the modification by the conformal factor. The transformation of the noise by the transformation of the power density data is illustrated in figure \ref{fig:noisedist} for the element $H_{44}$ using a noise level of $$\frac{\norm{\widetilde{H}_{44}-H_{44}}_{L^2(M)}}{\norm{H_{44}}_{L^2(M)}}\cdot 100=0.70.$$

\begin{figure}
    \centering
    \includegraphics[width=\textwidth]{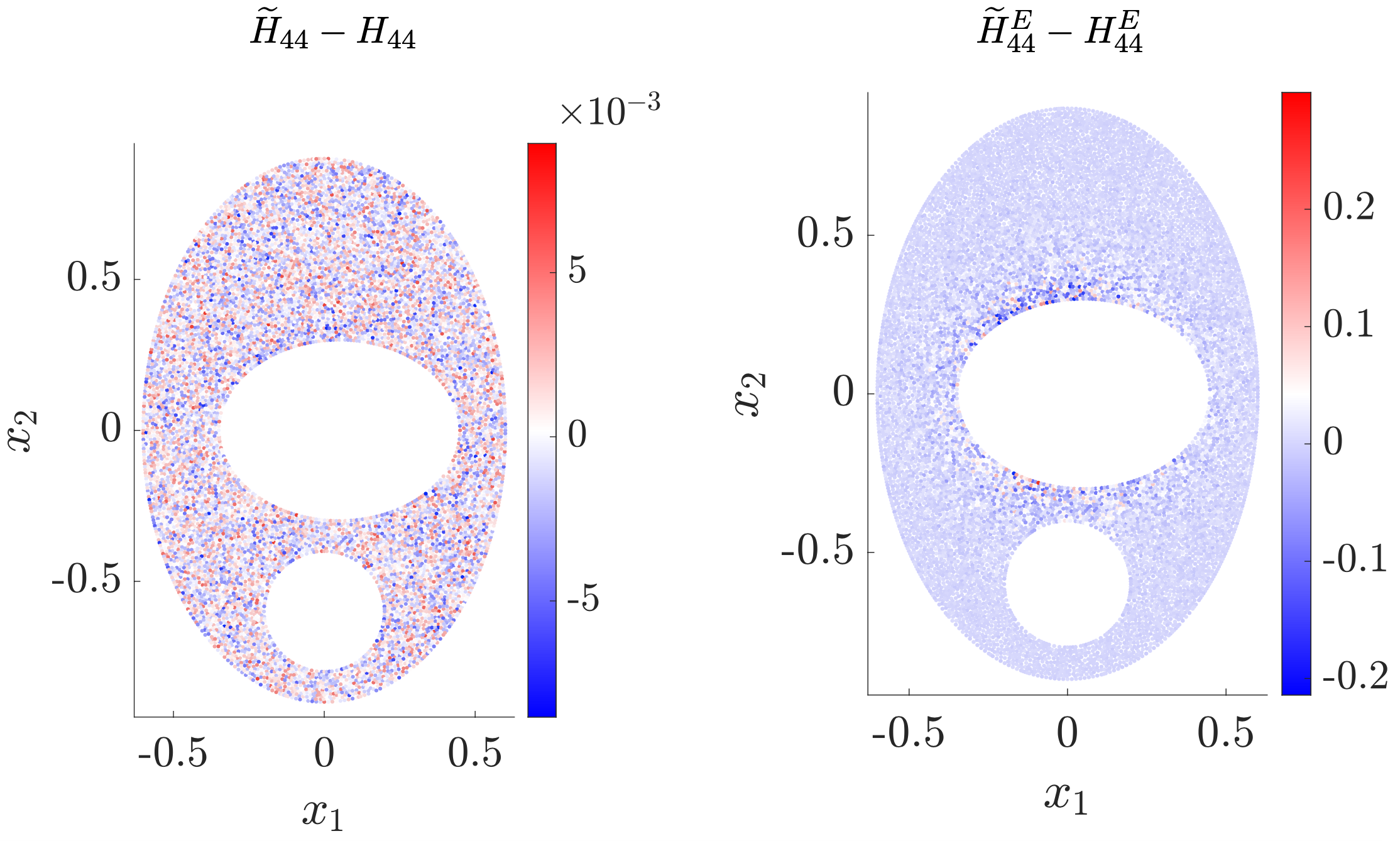}
    \caption{A noise level of 0.7\% influenced by the transformation of the power density data from the manifold $(M,g)$ to the plane $(M,g_E)$ inducing a noise level of 1.1\%. This is illustrated for the element $H_{44}$.}
    \label{fig:noisedist}
\end{figure}

    
The influence of the transformation by the square of the conformal factor $\rho$ is clearly visible (see an illustration of $\rho^2$ in figure \ref{fig:rho}), as for the Euclidean domain the noise is magnified by $\rho$ around the ellipse shaped hole in the center. For the Euclidean domain the noise level of $0.7\%$ corresponds to a noise level of $$\frac{\norm{\widetilde{H}_{44}^E-H_{44}^E}_{L^2(M)}}{\norm{H_{44}^E}_{L^2(M)}}\cdot 100=1.1.$$ Furthermore, we note that the noise level is no longer the same for all elements $H_{ij}^E$ as documented in table \ref{tab:Hnoise}.

\begin{table}[]
    \centering
    \begin{tabular}{c|c|c|c|c|c|}
         & $\widetilde{H}^E_{11}$ & $\widetilde{H}^E_{12}$ & $\widetilde{H}^E_{22}$ & $\widetilde{H}^E_{24}$ & $\widetilde{H}^E_{44}$ \\\hline
         Corresponding noise level& 1.2\% & 0.78\% & 1.0\% & 0.73\% & 1.1\% \\ \hline
    \end{tabular}
    \caption{Influence of a $0.7\%$ noise level of $\widetilde{H}$ on the noise level of $\widetilde{H}^E$.}
    \label{tab:Hnoise}
\end{table}

In the following we use the Euclidean reconstruction procedure to reconstruct $\gamma$ from noisy data $\widetilde{H}^E$ similarly to the previous section. However, we only add very small levels of noise for different reasons. The first reason is that adding noise on $(M,g)$ results in a higher noise level in $(M,g_E)$ which is differing from element to element of $H^E$ in all our simulations and it also results in a different distribution of the noise level. Furthermore, the reconstruction of $\xi$ and $\zeta$ is relatively unstable as already indicated by the reconstructions from noise free data in the previous section and this is also quantified in \cite[eq. 16]{BalMonard} (to the extent that the reconstruction of $\det(\gamma)^{\frac{1}{2}}$ is more stable than the reconstruction of $\tilde{\gamma}$). The last reason is that adding noise quickly results in that conditions \eqref{cond11E} and \eqref{cond12E} are violated. Especially reason three can be dealt with by using different regularization techniques as indicated in \cite{BalMonard}. However, we limit ourselves to only study the effect of this peculiar transformation of the noise on the reconstruction instead of addressing the previous issues. We are thus only able to add the extremely small amount of $0.0004\%$ noise on $(M,g)$ which corresponds to a noise level of $0.0004-0.0006\%$ in $(M,g_E)$ so that condition \eqref{cond12E} is still satisfied. The reconstructions are illustrated in the second row in figure \ref{fig:RecNoise4}.

\begin{figure}[h!]
    \centering
    \includegraphics[width=\textwidth]{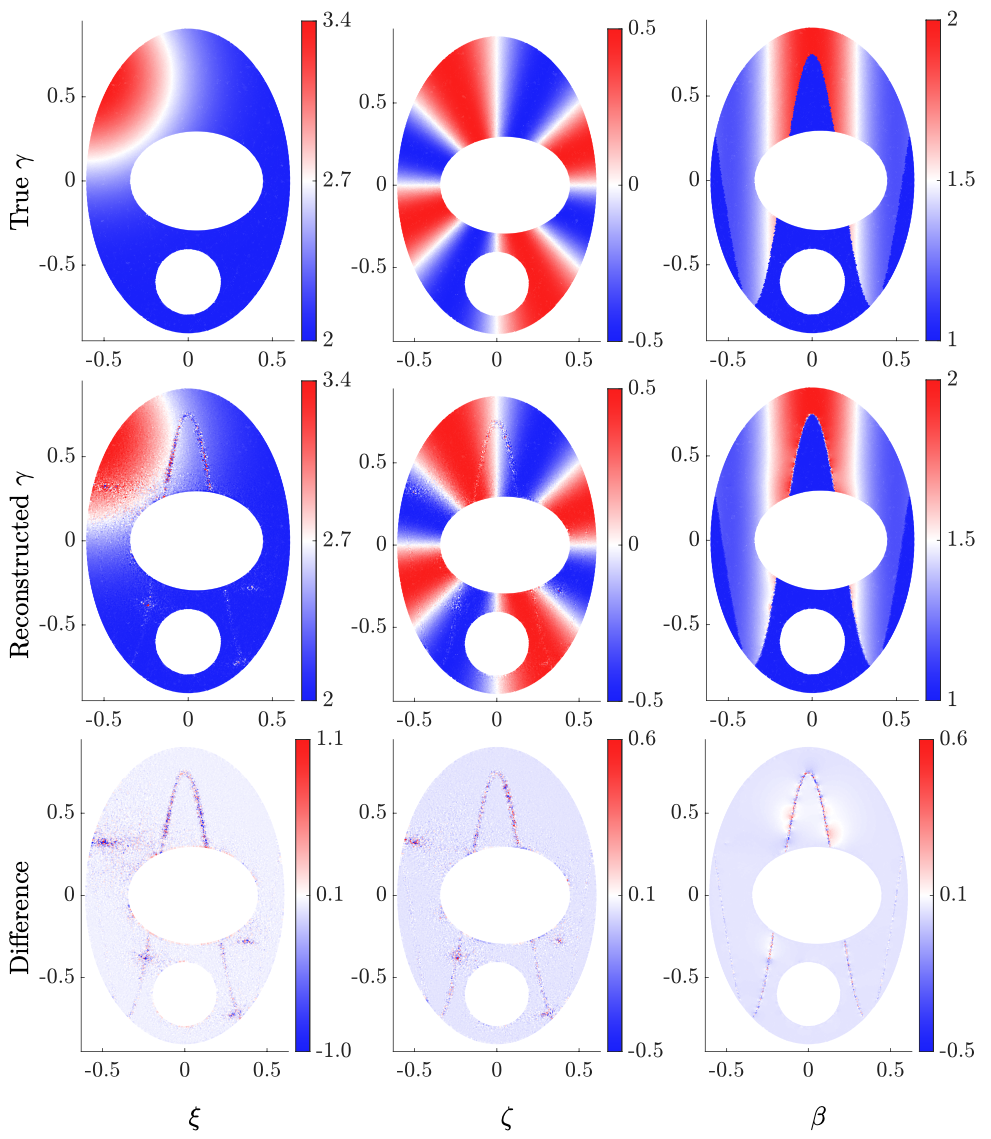}
    \caption{The true scalar functions $\xi$, $\zeta$ and $\beta$ determining the conductivity in $(M,g_E)$ in the first row and their reconstructions in presence of 0.0004\% noise (corresponding to 0.0004\%-$0.0006\%$ in $(M,g_E)$) in the second row. Their difference is shown in the third row.}
    \label{fig:RecNoise4}
\end{figure}
    
The relative $L_2$-errors are given by 4.8\% ($\xi$), 12.7\% ($\zeta$) and 2.11\% ($\beta$). Artifacts appear at similar locations to the noise free case, but the artifacts are spread out over larger regions in the presence of noise. Visually the biggest difference for the reconstructions in comparison to the noise free data can be seen for $\xi$ around the ellipse shaped hole. The transformation of the noise yields more artifacts around this boundary. The corresponding ellipse field of $\gamma$ on $(M,g)$ is very similar to the visualization in figure \ref{fig:RecCat} and for that purpose we omit it. However, in the comparison of the ellipses there is a visual difference between the reconstruction with noise and without noise so we illustrate the difference $m^2(\gamma,\gamma_{\text{true}})/\kappa$ in figure \ref{fig:DiffRecCatNoise}. The accumulated relative difference is given by:
\begin{equation*}
    \int_M \frac{m^2(\gamma,\gamma_{\text{true}})}{\kappa} \, \mathrm{d}V = 0.1484.
\end{equation*}
From the relative difference we see that the error between the true conductivity and $\gamma$ is larger. When comparing the figures there appear more artifacts in the reconstruction of $\gamma$ from noisy data and especially towards the lower boundary on the catenoid, which corresponds to the ellipse shaped hole in the parameter domain. So one can see the influence of how the magnification of the noise by $\rho$ around the ellipse shaped hole in $M$ results in more artifacts towards this boundary on $(M,g)$ in the reconstruction of $\gamma$. Furthermore, the point cloud left of the circular shaped hole is much bigger than in the noise free case. This point cloud appears, since condition \eqref{cond12M} is close to being violated here. Hence, the presence of noise makes it even harder to satisfy this condition and induces more values close to zero. 


\begin{figure}[h!]
    \centering
    \begin{minipage}[t]{\textwidth}
        \begin{minipage}[t]{\textwidth}
        \centering
        \includegraphics[width=0.6\linewidth]{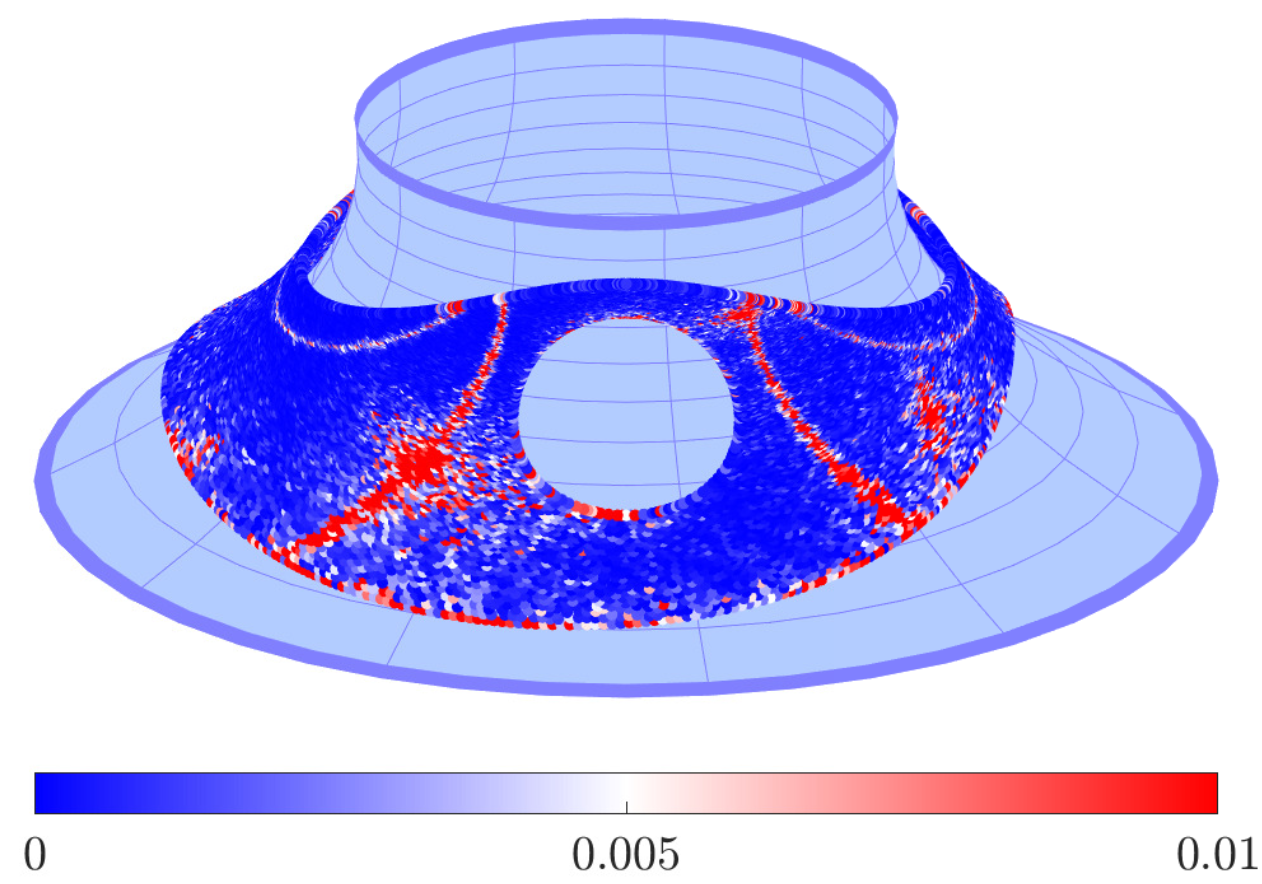}
    \end{minipage}
            \end{minipage}
    \caption{The coordinate invariant difference between the reconstructed conductivity, $\gamma$, in the presence of $0.0004\%$ noise and the true conductivity, $\gamma_{\text{true}}$, given by $m^2(\gamma,\gamma_{\text{true}})/\kappa$.}
    \label{fig:DiffRecCatNoise}
    \end{figure}
    
We note that while $\xi$ and $\zeta$ are highly affected by the low noise level of $0.0004\%$ in the Euclidean reconstruction procedure, there is almost no effect in the reconstruction for $\beta$. For this purpose we consider another simulation where we use the true anisotropy $\tilde{\gamma}$ composed of $\xi$ and $\zeta$ and add a noise level of $0.7\%$ on $(M,g)$ to see the effect of the noise on the reconstruction of $\beta$ resulting in the noise levels documented in table \ref{tab:Hnoise} for $\widetilde{H}^E$. We note again that it was not possible to add a higher noise level than that as this would have resulted in the violation of condition \eqref{cond11E}. In this case we don't need to use the second order basis functions, as we only need to meet condition \eqref{cond11E} and this reduces the mesh size to $N_1=23904$ nodes for the fine mesh and $N_2=10695$ for the coarser mesh. Furthermore, we use the coordinate functions as boundary conditions $(f_1,f_2)=(x^1,x^2)$ as this allows for adding a higher noise level while still satisfying condition \eqref{cond11E}. The reconstruction is illustrated in figure \ref{fig:RecNoisedet} and the relative $L_2$-error is 2.68\%. In this case the transformation of the noise by $\rho^2$ clearly induces artifacts appearing around the ellipse shaped hole. We consider $\gamma$ on the catenoid and illustrate the relative difference $m^2(\gamma,\gamma_{\text{true}})/\kappa$ in figure \ref{fig:DiffRecCatNoiseDet}. The accumulated relative difference is in this case given by:
\begin{equation*}
    \int_M \frac{m^2(\gamma,\gamma_{\text{true}})}{\kappa} \, \mathrm{d}V = 0.1006.
\end{equation*} 
Note that the accumulated relative difference is even less than when adding no noise or very little noise on all entries of $\gamma$. This is due to the fact that the true anisotropy was used for reconstruction and that the reconstruction of the determinant of $\gamma$ is more stable to noise. From the figure we see that since more noise was added it is even more significant how the transformation of the noise between $(M,g)$ and $(M,g_E)$ yields more artifacts towards the lower boundary of the catenoid. We note that in this case only the condition \eqref{cond11M} needs to be satisfied so there are no artifacts induced by points where condition \eqref{cond12M} is close to being violated.


\begin{figure}
    \centering
    \includegraphics[width=\textwidth]{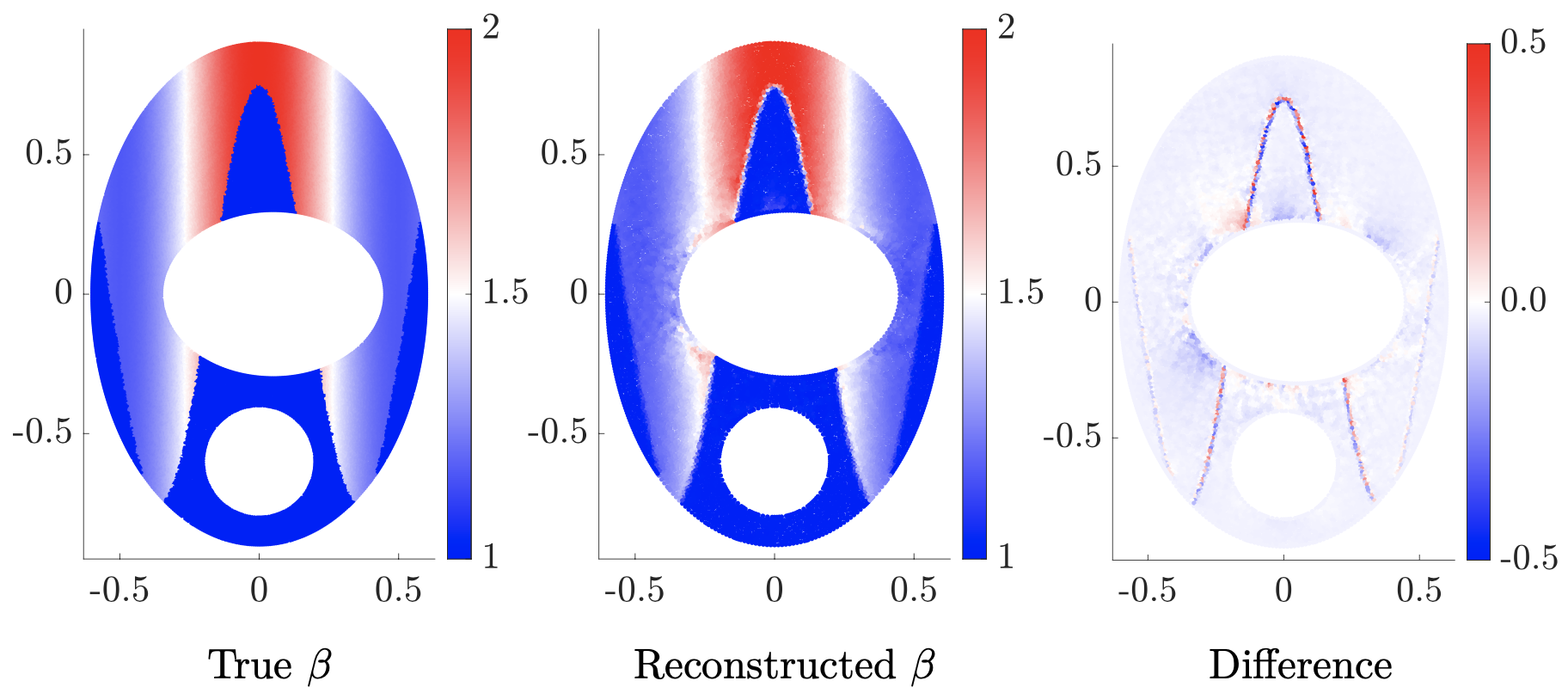}
    \caption{Using the true anisotropy $\tilde{\gamma}$ and reconstructing $\beta$ in the presence of 0.7\% noise (corresponding to 0.73\%-1.2\% noise in $(M,g_E)$).}
    \label{fig:RecNoisedet}
\end{figure}


\begin{figure}[h!]
    \centering
    \begin{minipage}[t]{\textwidth}
        \begin{minipage}[t]{\textwidth}
        \centering
        \includegraphics[width=0.6\linewidth]{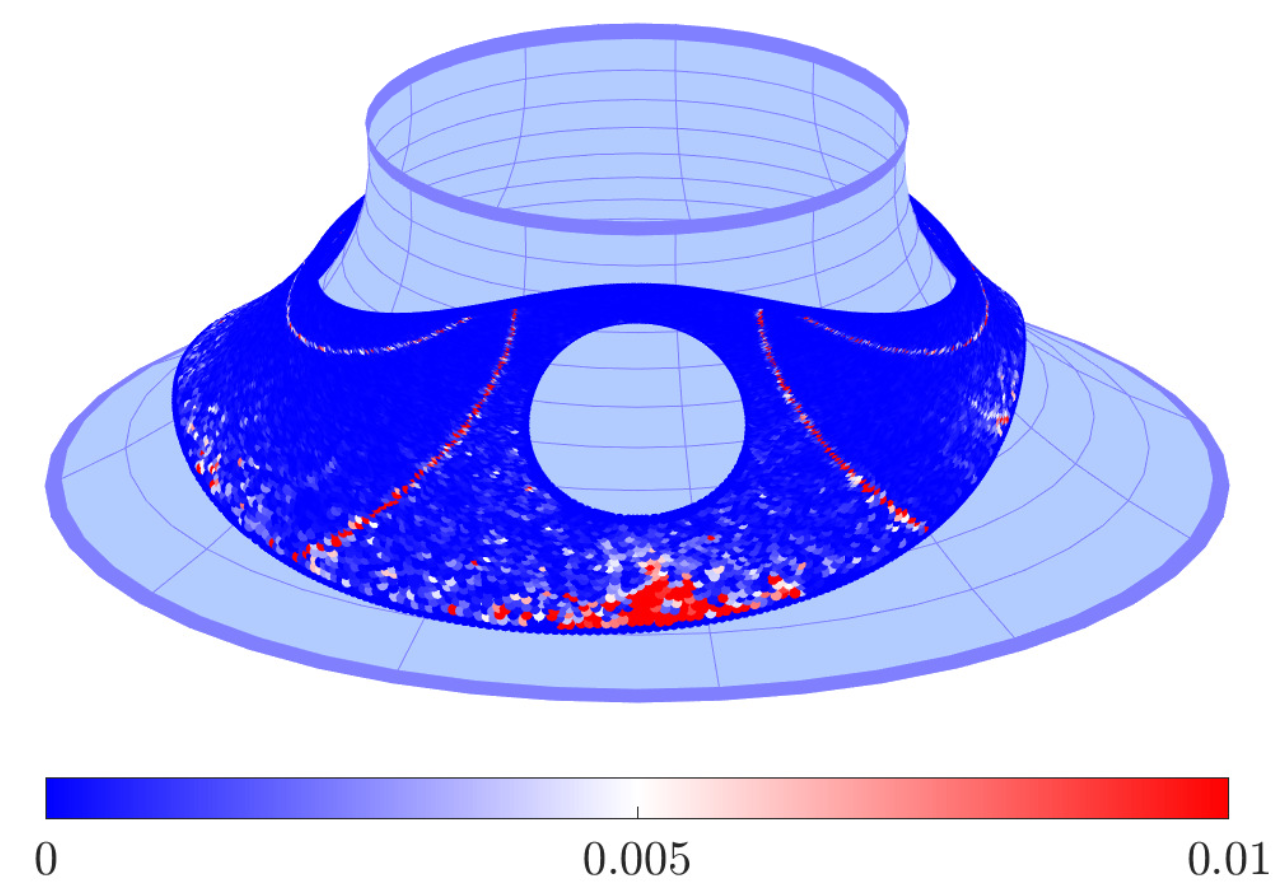}
    \end{minipage}
            \end{minipage}
    \caption{The coordinate invariant difference between the reconstructed conductivity, $\gamma$, in the presence of $0.7\%$ noise and using the true anisotropy $\tilde{\gamma}$ so that only $\beta$ is unknown. The difference is computed by $m^2(\gamma,\gamma_{\text{true}})/\kappa$.}
    \label{fig:DiffRecCatNoiseDet}
    \end{figure}

\subsection{Other parametrizations}
\subsubsection{A non-conformal parametrization}
As alluded to in section \ref{secRecProcNonConf} in general it might not be straightforward to find a conformal parametrization. This section simulates the setting for Algorithm \ref{algoM}. We start with a non-conformal parametrization of the catenoid in figure \ref{fig:cat} and determine a diffeomorphism that yields the conformal parametrization in \eqref{eqConfParam}. The parameter domain $N$ for the non-conformal parametrization is illustrated in figure \ref{fig:ParameterDomainGen}. 


\begin{figure}[h!]
    \centering
    \includegraphics[width=0.5\textwidth]{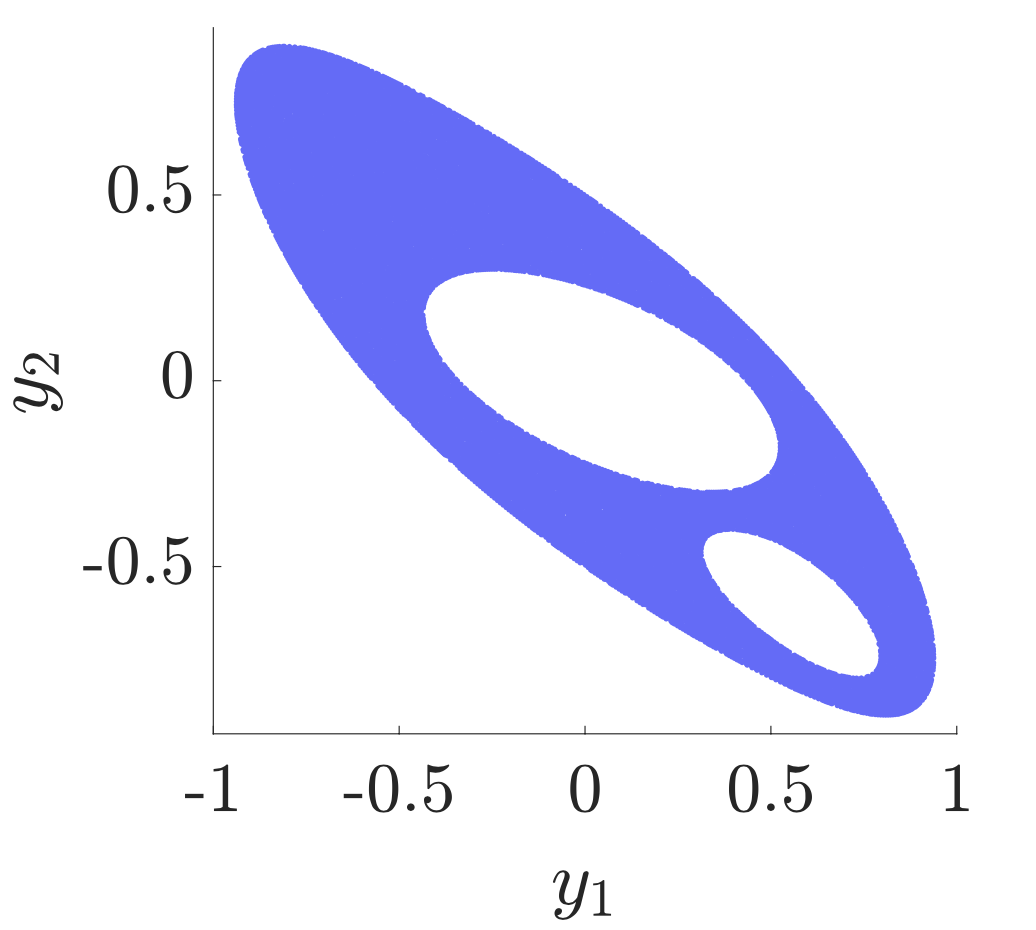}
    \caption{Parameter domain $N$ for the non-conformal parametrization of the catenoid.}
    \label{fig:ParameterDomainGen}
\end{figure}

Now consider the following parametrization $r_N$ of $(N,g_N)$ with corresponding metric matrix function $G_N$:

\begin{align*}
    r_N(y^1,y^2) &= \begin{pmatrix}
    \dfrac{(\sinh(y^1)+y^2)\,\omega(y^1,y^2)}{\alpha(y^1,y^2)}\\
    \dfrac{y^2\,\omega(y^1,y^2)}{\alpha(y^1,y^2)}\\
      \frac{1}{2} \log \lp \alpha(y^1,y^2)^2 \rp \\
    \end{pmatrix},\\
    G_N(y_1,y_2) &= \frac{\omega(y^1,y^2)^2}{\beta(y^1,y^2)}\begin{bmatrix} \cosh(y^1)^2 & \cosh(y^1) \\ \cosh(y^1) & 2  \end{bmatrix}\\
\intertext{where}
    \alpha(y^1,y^2)&=\sqrt{ \sinh(y^1)^2+2\sinh(y^1) y^2+2(y^2)^2} ,\\
    \beta(y^1,y^2)&=\cosh(y^1)^2+2\sinh(y^1)y^2+2(y^2)^2 - 1,\\
    \omega(y^1,y^2)&=\cosh \lp \frac{1}{2} \log \lp \alpha(y^1,y^2)^2 \rp \rp.
\end{align*}

The three boundary conditions that give rise to the corresponding power density data as in section \ref{secNumConf} is then given by  $$(f_{N,1},f_{N,2},f_{N,4})=(-y^2-0.1 (y^2)^2,\sinh(y^1),0.1(y^2)^2-0.1\sinh^2(y^1)+y^2)$$ (the third boundary condition is $f_{N,3}=f_{N,2}$.) The corresponding solutions$(u_{N,1},u_{N,2},u_{N,3},u_{N,4})$ satisfy according to Proposition \ref{prop:RelatingConditions} the conditions \eqref{cond11M}-\eqref{cond12M} and are used to construct the power density matrix $H^N$. The element $H_{44}^N$ is illustrated in figure \ref{fig:PowDen}. In the following we go through the steps in algorithm \ref{algoM} in order to reconstruct $\gamma_N$ from $H^N$. \\

In practice one needs to follow the procedure by \cite{gu2002a,gu2004a,schoen1997a} to obtain a diffeomorphism $\psi: N \rightarrow M$ that yields a global conformal parametrization of the catenoid. However in this case we already know an explicit diffeomorphism $\psi$ that gives the desired result: 
 \begin{align*}
     \psi(y^1,y^2) &= (x^1,x^2) = (\sinh(y^1)+y^2,y^2),\\
     \intertext{with inverse}
     \psi^{-1}(x^1,x^2) &= (y^1,y^2) = (\text{arcsinh}(x^1-x^2),x^2).
 \end{align*}
This yields the parametrization $r_M(x^1,x^2)=r_N(\psi^{-1}(x^1,x^2))$ as in equation \eqref{eqConfParam}.
The Jacobi matrices corresponding to $\psi$ and $\psi^{-1}$ are defined as follows:
\begin{align*}
    D\psi(y^1,y^2)&=\begin{bmatrix}
        \cosh(y^1) & 1\\ 0 & 1
    \end{bmatrix},\\
    D\psi^{-1}(x^1,x^2)&= \begin{bmatrix}
        \frac{1}{\sqrt{(x^1)^2}-2x^1 x^2 + (x^2)^2+1} & -\frac{1}{\sqrt{(x^1)^2}-2x^1 x^2 + (x^2)^2+1}\\ 0 & 1
    \end{bmatrix}.
\end{align*}
From the calculation $(D\psi^{-1})^t(x) G_M(\psi^{-1}(x)) D\psi^{-1}(x)$ we arrive at the conformal factor $\rho$ as in equation \eqref{eqConfFactor} so that $G_M(x^1,x^2)=\rho^2(x^1,x^2) \, G_E$ as desired. The transformation of the power density data from $H^N$ to $H^M$ is illustrated in figure \ref{fig:PowDenGen} for the element $H_{44}$. We note that since $H$ is invariant with respect to coordinate changes, $H^M$ and $H^N$ are identical and only differ in the appearance of the parameter domain. The remaining steps in the reconstruction procedure are identical to the approach in section \ref{secNumConf} in order to arrive at the reconstructed conductivity as in figure \ref{fig:RecCat}.


\begin{figure}
    \centering
    \includegraphics[width=0.9\textwidth]{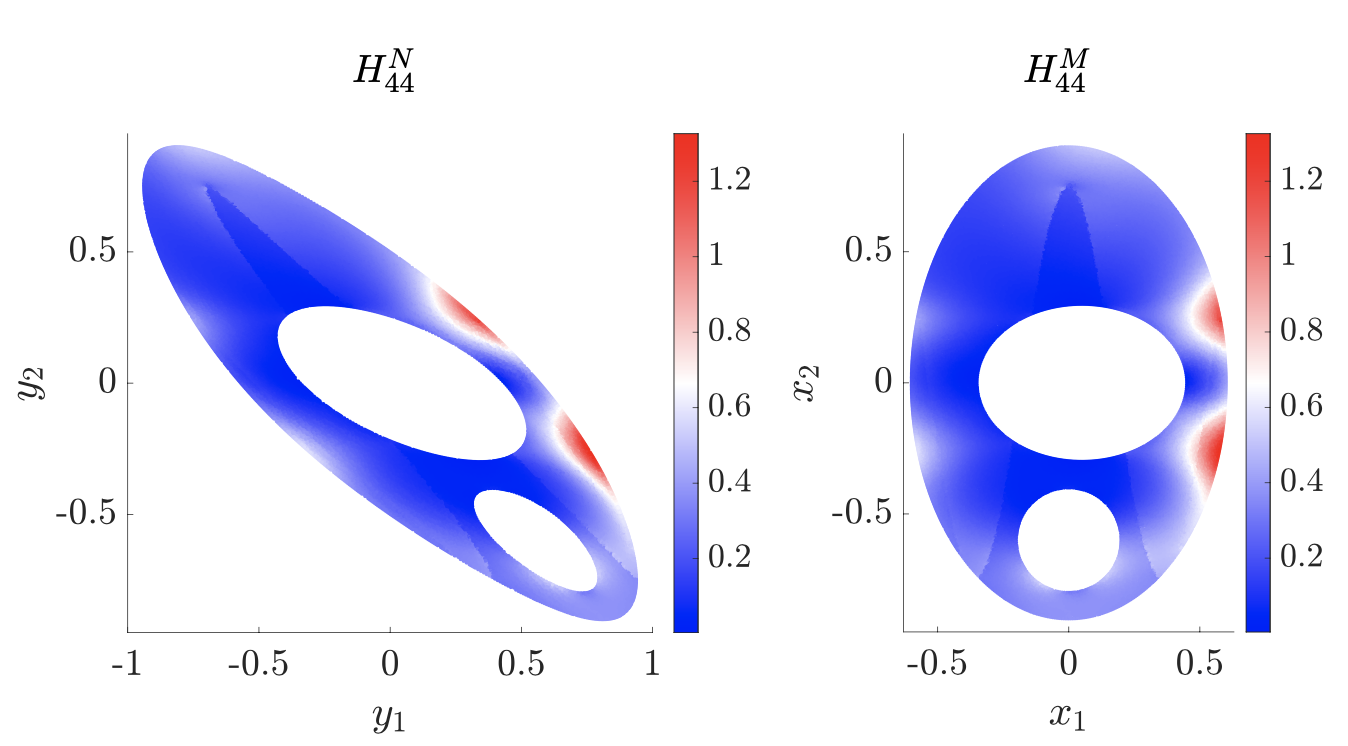}
    \caption{The transformation of the power density data from $(N,g_N)$ to $(M,g_M)$ for the element $H_{44}$.}
    \label{fig:PowDenGen}
\end{figure}

\subsubsection{A conformal parametrization with a periodic parameter domain}
Finally, we consider another conformal parametrization $(\widetilde{M},g_{\widetilde{M}})$ of the catenoid in figure \ref{fig:cat}. This is the standard textbook conformal parametrization of the catenoid:
\begin{equation}
 r_{\widetilde{M}}(\tilde{x}^{1}, \tilde{x}^{2})  =   (\cosh(\tilde{x}^{1}) \, \cos(\tilde{x}^{2}) , \cosh(\tilde{x}^{1})\,\sin(\tilde{x}^{2}) ,  \tilde{x}^{1}), \quad \tilde{x}^1\in \mathbb{R}, \quad \tilde{x}^2\in ]-\pi,\pi[.
\end{equation}
We illustrate the corresponding parameter domain in figure \ref{fig:TildeM}. Since the cosine and sine functions enter the parametrization $r_{\widetilde{M}}$  a periodic boundary appears when $\tilde{x}^2$ shifts from $-\pi$ to $\pi$ (and the other way around). If we compare this parameter domain to $M$ corresponding to the conformal parametrization of the catenoid this corresponds to "cut open" $M$ along a straight line. We illustrate $M$ with that line highlighted in red in figure \ref{fig:TildeM}. 


\begin{figure}
    \centering
    \includegraphics[width=0.9\textwidth]{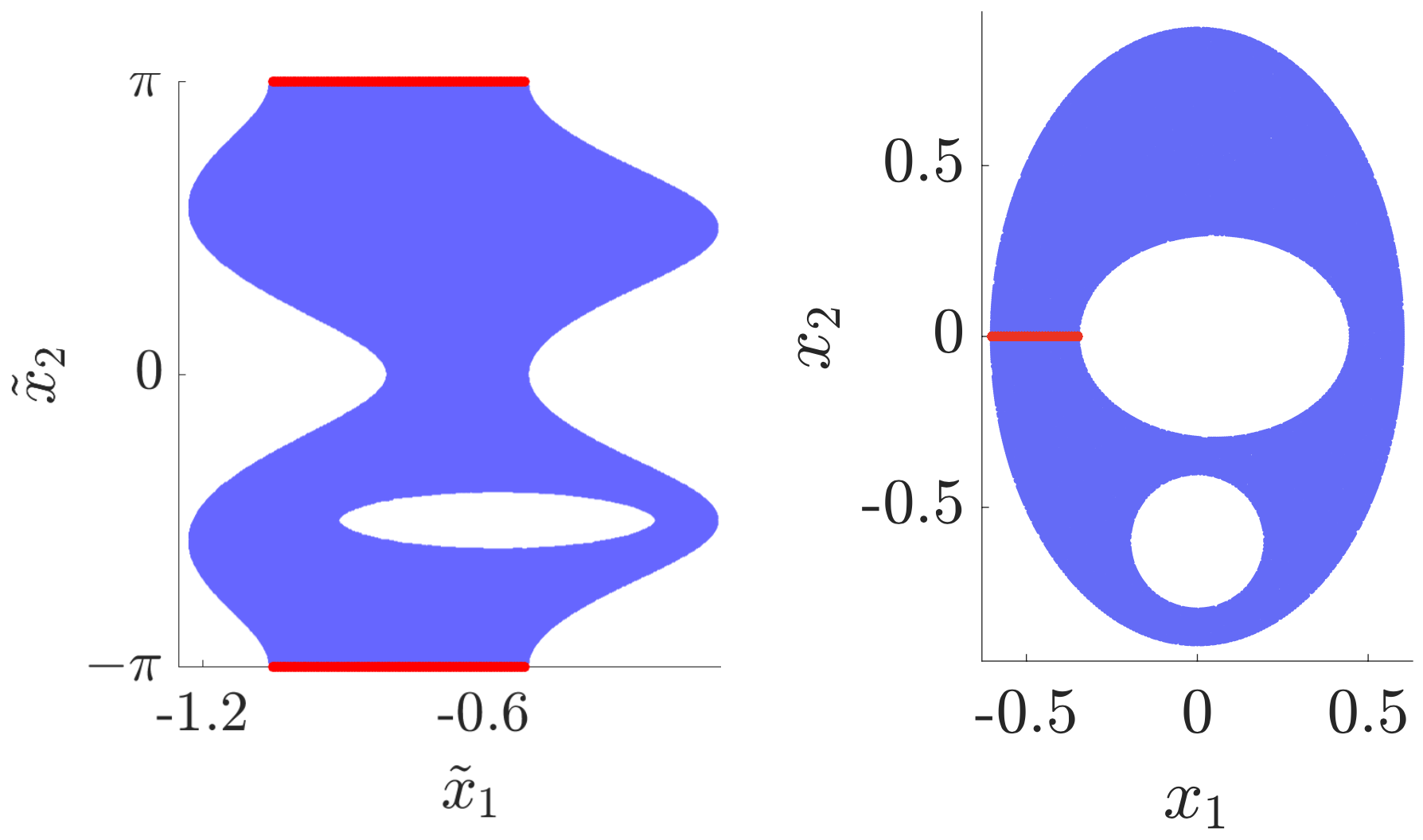}
    \caption{The parameter domains $\widetilde{M}$ and $M$ for the conformal parametrizations $r_{\tilde{M}}$ and $r_M$. The red part in $\widetilde{M}$ indicates the periodic boundary corresponding to "cut open" $M$ along the red line.}
    \label{fig:TildeM}
\end{figure}

The theory developed for the Euclidean reconstruction procedure as considered in the proof of theorem \ref{MainThm} relies on the fact that the domain has no periodic boundaries. For this setting it is guaranteed that there exist four boundary functions so that the conditions \eqref{cond11E}-\eqref{cond12E} are satisfied and it is possible to reconstruct the conductivity. So the theory allows one to reconstruct $\gamma_M$ from power densities when considering the parametrization $r_M$. However, as there exists a diffeomorphism $\phi$ that maps from the parameter domain $\widetilde{M}$ to $M$ this implies that 
in this particular case a suitable modification of the reconstruction procedure should work equally well in the $\tilde{M}$-setting. This diffeomorphism $\phi$ is defined explicitly by:
 \begin{align*}
     \phi(\tilde{x}^1,\tilde{x}^2) &= (x^1,x^2) = \lp\log \lp \sqrt{(\tilde{x}^1)^2+(\tilde{x}^2)^2} \rp,\arg(\tilde{x}^1 + i\tilde{x}^2)\rp,\\
     \intertext{with inverse}
     \phi^{-1}(x^1,x^2) &= (\tilde{x}^1,\tilde{x}^2) = (e^{x^1}\cos(x^2),e^{x^1}\sin(x^2)).
 \end{align*}
In this work we have limited ourselves to study the setting without periodic boundaries. However, in future work this should definitely be considered
as an integral part of a proof of the full conjecture \ref{Conj:General} stated in the introduction. Especially, because generalizations to higher genus surfaces require that one is able to deal with periodic boundaries both in theory and for the numerical simulations. The periodic problem is closely related to the limited view setting considered in \cite{jensen2023a}.


\section{Conclusion}
\label{secCon}
We have presented a new geometric setting for the reconstruction of anisotropic conductivities from power densities. Our main result generalizes the reconstruction method for the 2-dimensional Euclidean setting to  2-dimensional compact Riemannian manifolds with genus 0. The result is presented in a way that opens for further research in the setting of Riemannian manifolds with higher genus (in $2$-dimensions) and possibly in higher dimensions as well. The approach applies to other similar inverse problems with internal data, in particular the reconstruction problem for anisotropic conductivities from \emph{current densities}, c.f. \cite{bal2014a,bal2014b}.

\clearpage
\printbibliography

\end{document}